\documentclass[11pt,reqno]{article}

\usepackage{amsmath,amsthm,amssymb}

\usepackage{stmaryrd}
\usepackage{fullpage}

\usepackage{hyperref}
\usepackage{color}

\usepackage{graphicx}
\usepackage[numbers]{natbib}

\usepackage{authblk}

\newcommand{\BBE}{\mathbb{E}}
\newcommand{\BBC}{\mathbb{C}}
\newcommand{\BBR}{\mathbb{R}}

\newcommand{\fPois}{\mathfrak{p}}

\newcommand{\BBe}{\mathbf{e}}

\newcommand{\cA}{\mathcal{A}}

\newcommand{\cS}{\mathcal{S}}

\newcommand{\cP}{\mathcal{P}}
\newcommand{\cD}{\mathcal{D}}

\newcommand{\Id}{\mathbb{I}}

\newcommand{\Ave}{\mathbf{A}}
\newcommand{\Har}{\mathbf{H}}

\DeclareMathOperator{\Tr}{Tr}

\DeclareMathOperator{\diff}{d\!}

\newtheorem{theorem}{Theorem}

\newtheorem{proposition}{Proposition}
\newtheorem{lemma}{Lemma}

\theoremstyle{definition}
\newtheorem{definition}{Definition}

\newtheorem{remark}{Remark}

\title{Matrix Means and a Novel High-Dimensional Shrinkage Phenomenon}

\author[1]{Asad Lodhia}
\author[2]{Keith Levin}
\author[3]{Elizaveta Levina}
\affil[1]{Broad Institute of MIT and Harvard}
\affil[2]{Department of Statistics, University of Wisconsin-Madison}
\affil[3]{Department of Statistics, University of Michigan}
\date{}

\begin{document} 
\maketitle

\begin{abstract}
Many statistical settings call for estimating a population parameter,
most typically the population mean, based on a sample of matrices.
The most natural estimate of the population mean is the arithmetic mean,
but there are many other matrix means that may
behave differently, especially in high dimensions.
Here we consider the matrix harmonic mean as an alternative
to the arithmetic matrix mean.
We show that in certain high-dimensional regimes,
the harmonic mean yields an improvement
over the arithmetic mean in estimation error as measured by
the operator norm.
Counter-intuitively, studying the asymptotic behavior of these two matrix
means in a spiked covariance estimation problem,
we find that this improvement in operator norm error does not imply
better recovery of the leading eigenvector.
We also show that a Rao-Blackwellized version of the harmonic mean is equivalent to a linear shrinkage estimator studied previously in the high-dimensional covariance estimation literature, while applying a similar Rao-Blackwellization to regularized sample covariance matrices yields a novel nonlinear shrinkage estimator.
Simulations complement the theoretical results, illustrating the conditions
under which the harmonic matrix mean yields an empirically better estimate.
\end{abstract}

\section{Introduction} \label{sec:intro}

Matrix estimation problems arise in statistics in a number of areas,
most prominently in covariance estimation, but also in network analysis,
low-rank recovery and time series analysis.
   %Commonly the focus is on estimating a single matrix; for example, the covariance estimation literature (see review by \liza{please add citation to Fan et al}) focuses on estimating one covariance matrix from a sample of vectors;  network analysis and low rank recovery focus on estimating the expectation of a matrix \liza{(cite something?  Chatterjee as a generic reference?  or nothing is fine too)}, under some structural assumptions, from a single noisy realization.  However, there are situations where a sample of matrices is collected, and the goal is to estimate the underlying common population mean of these matrices.
   Typically, the focus is on estimating a matrix based on a single noisy realization of that matrix.
   For example, the problem of covariance estimation \cite{FanLiaLiu2016} focuses on estimating a population covariance matrix based on a sample of vectors, which are usually combined to form a sample covariance matrix or another estimate of the population covariance. In network analysis and matrix completion problems, the goal is typically to estimate the expectation of a matrix-valued random variable based on a single observation under suitable structural assumptions (see, e.g., \cite{Chatterjee2015} and citations therein). A related setting that has received less attention is the case where a sample of matrices is available and the goal is to estimate an underlying population mean or other parameter. This arises frequently in neuroimaging data analysis, where each matrix represents connectivity within a particular subject's brain and the goal is to estimate a population brain connectivity pattern \cite{BulSpo2009,Sporns2012}.
 
The most direct approach to estimating the underlying population mean from a sample of matrices is to take the arithmetic (sample) mean, perhaps with some regularization to ensure the stability of the resulting estimate.  The arithmetic matrix  mean is the mean with respect to Euclidean geometry on the space of matrices, which is often not the most suitable average for a given matrix model.  A simple example can be found in our recent work  \cite{LevLodLev2019}, where we showed that in the problem of estimating a low-rank expectation of a collection
of independent random matrices with different variances, a weighted average improves upon the na\"ive arithmetic matrix mean,
analogously to the scalar analogue, in which a weighted average can improve
upon the unweighted mean in the presence of heterogeneous variances.
Somewhat surprisingly in light of the rich geometry of matrices, fairly little attention has been paid in the literature to other matrix geometries and their associated means.
An exception is work by Schwartzman \cite{Schwartzman2016}, who
argued for using the intrinsic geometry of the positive definite cone
\cite{Bhatia2006} in the problem of covariance estimation,
and showed that a mean with respect to this different matrix geometry can, under certain models, yield an appreciably better estimate.   Recent work has also considered Fr\'{e}chet means in the context of multiple-network analysis \cite{LunOlhWol2019}.  Continuing in this vein, the current work aims to better understand the sampling distribution of matrix means other than the arithmetic mean under different matrix geometries.

Computing means with respect to different
geometries has been studied at some length in the signal processing and
computer vision communities, mostly in the context of the Grassmann and
Stiefel manifolds \cite{AbsMahSep2004,EdeAriSmi1998,MarBevDraKirPet2014}.
% AbsMahSep2004
% https://link-springer-com.proxy.lib.umich.edu/article/10.1023/B:ACAP.0000013855.14971.91
% EdeAriSmi1998
% https://www.cis.upenn.edu/~cis515/Stiefel-Grassmann-manifolds-Edelman.pdf
% MarBevDraKirPet2014.
% https://ieeexplore-ieee-org.proxy.lib.umich.edu/document/6909538
See \cite{Smith2005} for a good discussion of how taking intrinsic
geometry into account leads to estimators other than the
arithmetic mean. % http://ssg.mit.edu/cal/abs/2005_spring/stsmith.pdf 
Recent work has considered similar geometric concerns
in the context of network data
\cite{KolLinRosXuWal2017,LunOlhWol2019}.
Kolaczyk and coauthors \cite{KolLinRosXuWal2017} considered the problem
of averaging multiple networks on the same number of vertices,
developed a novel geometry for this setting and derived a Fr\'{e}chet
mean for that geometry.
Recent work by Lunag\'{o}mez, Olhede and Wolfe \cite{LunOlhWol2019}
considered a similar network averaging problem,
and presented a framework for both specifying distributions of networks
and deriving corresponding sample means.
Unfortunately, most of these matrix means are not amenable to the currently available tools from random matrix theory that could help analyze their properties.

In this paper, we consider the behavior of the harmonic mean of a collection of random matrices, a matrix mean that arises from a markedly different eometry than the matrix arithmetic mean.
The harmonic mean turns out to be well-suited to analysis using techniques in random matrix theory, and it is our hope that results established here will be extended to other related matrix means in the future.
Building on random matrix results developed by the first author \cite{Lod19}, we show how the harmonic matrix mean can, in certain regimes, yield a better estimate of the population mean matrix in spectral norm compared to the arithmetic mean.
We also show that this improvement does not carry over to recovery of the top population eigenvector in a spiked covariance model, making an important distinction between two measures of estimation performance that are often assumed to behave similarly.
We characterize the settings in which the harmonic matrix mean improves upon the arithmetic matrix mean as well as the settings in which it does not, and show the implications of these results for covariance estimation. 
% Finally, we derive the Rao-Blackwellized version of the harmonic mean of two random covariance matrices and draw connections between the resulting estimator and a family of regularized covariance estimators that has been previously studied in the literature. 

Our focus in this work is on estimating the population mean of a collection of Wishart matrices, which can be thought of as sample covariance matrices.
There is an extensive literature on estimating a population covariance matrix on $p$ variables from $n$ observations based on a single covariance matrix.
The sample covariance matrix is the maximum likelihood estimator for Gaussian data, and when the dimension $p$ is fixed, classical results fully describe its behavior \cite{Muirhead,Anderson}.
In the high-dimensional regime, where $p$ is allowed to grow with $n$,  the sample covariance matrix is not well-behaved, and in particular becomes singular as soon as $p \ge n$.
There has been extensive work on understanding this phenomenon in random matrix theory, starting from the pioneering work of \cite{MarPas1967} and followed by numerous more recent results, especially focusing on estimating the spectrum \cite{YaoZheBai2015, ElKaroui2008}.
Much work in random matrix theory has focused on the spiked model, in  which the population covariance is the sum of the identity matrix and a low-rank signal matrix. \cite{JohLu2009,BerRig2013,CaiMaWu2015,FanJohSun2018}.
% http://www-math.mit.edu/~rigollet/PDFs/BerRig12.pdf
% https://www.tandfonline.com/doi/abs/10.1198/jasa.2009.0121
% http://www-stat.wharton.upenn.edu/~tcai/paper/Spiked-Covariance.pdf
% https://arxiv.org/pdf/1806.09529.pdf

The problem of covariance estimation is now several decades old
(see, e.g., \cite{Stein1986}, and refer to \cite{DonGavJoh2018} for a thorough overview of early work).
In the past two decades, literature on covariance estimation in high dimensions (see~\cite{FanLiaLiu2016} for a review) has focused on addressing the shortcomings of the sample covariance, mainly by applying regularization.
 James-Stein type shrinkage was considered in early 
 Bayesian approaches~\cite{DanKas2001}
 %\cite{Stein1956,JamSte1961}.
% https://projecteuclid.org/download/pdf_1/euclid.bsmsp/1200501656
% https://projecteuclid.org/download/pdf_1/euclid.bsmsp/1200512173
%https://www.ncbi.nlm.nih.gov/pmc/articles/PMC2748251/
and in the Ledoit-Wolf estimator~\cite{LedWol2004},
which shrinks the sample covariance matrix towards the identity matrix using
coefficients optimized with respect to a Frobenius norm loss.
%$(1-\lambda)S + \lambda \nu I$, where $S$ is the sample covariance and $\lambda,\nu$ are scalars. The authors identified the optimal choices for $\lambda$ and $\nu$ under a Frobenius norm loss, and proved that replacing these with appropriate plug-in estimates yields an estimator that converges in squared Frobenius norm to the optimal estimator among all linear combinations of the sample covariance and the identity. 
Subsequent papers presented variants with different estimates of the optimal coefficients and different choices of the shrinkage target matrix %$\nu I$ under the special case of 
for normal data~\cite{CheWieEldHer2010,FisSun2011},
%http://xsun.people.clemson.edu/2011CSDA.pdf
%https://www.sciencedirect.com/science/article/pii/S0167947310004743
%https://web.eecs.umich.edu/~hero/SPreprints/chen_icassp1_09.pdf
as well as for the more general case of finite fourth
moments~\cite{Touloumis2015}.
%https://arxiv.org/pdf/1410.4726.pdf
A related line of work has focused on estimation with respect to Stein's loss
\cite{DeySri1985,Stein1986,LedWol2018}.

More recent work introduced the class of orthogonally invariant estimators
\cite{DonGavJoh2018}, which unifies many of the regularization approaches
outlined above.
Given a covariance matrix, an orthogonally invariant estimator
produces an estimate of the population covariance matrix
with eigenspace identical to that of the sample covariance,
and spectrum that is a function $\Phi$
of the spectrum of the sample covariance.
Donoho and coauthors \cite{DonGavJoh2018} showed how choices
of loss function yield different (asymptotically) optimal choices of $\Phi$.
Nonlinear eigenvalue regularization approaches in this vein have been explored
extensively \cite{ElKaroui2008,RaoMinSpeEde2008,LedWol2012,WonLimKimRaj2013,LiCheQinBaiYao2013,Lam2016,KonVal2017}.
Typically, the eigenvalues of the sample covariance matrix are adjusted
either according to the asymptotic relation between
the limiting spectral distribution of a spiked covariance
and its Stieltjes transform
or based on the method of moments applied to the limiting spectral distribution.
Approaches such as these tend to improve upon the linear shrinkage estimates
introduced in \cite{LedWol2004} by accounting for nonlinear dispersion of the sample eigenvalues with respect to their population analogues.

One shortcoming of orthogonally invariant estimators is that they do not in any way change 
the estimated eigenvectors, which are not consistent in the high-dimensional regime  \cite{JohLu2009,BGR11}.
An alternative approach that overcomes this limitation is to regularize the sample covariance matrix by imposing structural constraints.
This class of methods includes banding or tapering of the covariance matrix, suitable when the variables have a natural ordering \cite{WuPou2003,BicLev2008taper}, and thresholding when the variables are not ordered \cite{BicLev2008thresh, RotLevZhu2009,CaiLiu2011}.
Minimax rates for many structured covariance estimation
problems are now known \cite{CaiZhaZho2010,CaiZho2012sparse,CaiZho2012ell1};
see \cite{CaiRenZho2016} for an overview of these results.

%\subsection*{Roadmap}
The remainder of this paper is laid out as follows.
In Section~\ref{sec:setup} we establish notation and introduce the
random matrix models under study.
In Section~\ref{sec:datasplit}, we establish the
asymptotic behavior of the harmonic matrix mean under these models.
In Section~\ref{sec:raoblackwell} we compute a Rao-Blackwellized version
of the harmonic mean of two random covariance matrices, illuminating connections between the harmonic mean and a family of regularized covariance estimators.   
In Section~\ref{sec:vectorrecover}, we analyze a spiked covariance model and the behavior of the top eigenvector of the harmonic mean estimator under that model.  
%compute the limit of the inner product of the top eigenvector of the harmonic mean with the top eigenvector of the population covariance.
Finally, Section~\ref{sec:expts} briefly presents numerical simulations
highlighting the settings in which the harmonic matrix mean does and does not have an advantage in covariance estimation. Section~\ref{sec:discussion} concludes with discussion.

\section{Problem Setup} \label{sec:setup}
We begin by establishing notation.  
We denote the identity matrix by $\Id$, with its dimension clear from context.
For a $p\times p$ matrix $M$,
$\|M\|$ denotes its operator norm and
$\|M\|_F$ denotes its Frobenius norm.
For a set $A$, let
$\mathbf{1}_A(x) = 1$ if $x \in A$ and $\mathbf{1}_A(x) = 0$ otherwise. 
We denote by $\cS_p(\BBR)$ and $\cS_p(\BBC)$ the spaces of
$p \times p$ symmetric and Hermitian positive definite matrices, respectively.
For a $p\times p$ symmetric or Hermitian matrix $M$, the eigenvalues of $M$
are denoted $\lambda_1(M) \ge \lambda_2(M) \ge \dots \ge \lambda_p(M)$
and their corresponding eigenvectors are denoted $v_1(M),\dots, v_p(M)$.
We use $\preceq$ for the positive semidefinite ordering,
so that $M_1 \preceq M_2$ if and only if $M_2 - M_1$ is positive semidefinite.

Suppose that we wish to estimate the population mean $\Sigma$ 
of a collection of $N$ independent identically distributed self-adjoint 
positive definite $p$-by-$p$ random matrices.
The most commonly used model for positive (semi)definite random
matrices is the Wishart distribution, which arises in
covariance estimation and is well-studied in the random matrix theory
literature.

\begin{definition}[Wishart Random Matrix: Real Observation Model]
\label{def:realdatmodel}
Let $X$ be a random $p \times n$ matrix with columns drawn i.i.d.\ from
a centered normal with covariance $\Sigma \in \BBR^{p \times p}$.
Then
\begin{equation*}
W = \frac{ X X^* }{ n }
\end{equation*}
is a real-valued random Wishart matrix with parameters $\Sigma$ and $n$.
\end{definition}

Many of our results are also true 
for the complex-valued version of the Wishart distribution,
which we define here for the special case of identity covariance.

\begin{definition}[Wishart Random Matrix: Complex Observation Model]
\label{def:datamodel}
Let $X$ be a random $p \times n$ matrix with 
i.i.d.\ complex standard Gaussian random entries,
i.e., entries of the form
\[
\frac{Z_1 + \sqrt{-1} Z_2}{\sqrt{2}},
\]
where $Z_1$ and $Z_2$ are independent standard real Gaussian random variables.
Then $W = X X^*/n$ is a random matrix following the complex Wishart
distribution with parameters $\Id$ and $n$.
\end{definition}

Let $\{X_i\}_{i=1}^N$ be a sequence of independent identically
distributed $p \times n$ matrices with columns drawn i.i.d.\ from
a centered normal with covariance $\Sigma$.
Then for each $i=1,2,\dots,N$,
\begin{equation}
W_i:=\frac{X_i X_i^*}{n}
\end{equation}
is the sample covariance matrix, which follows the real-valued Wishart
distribution with parameters $\Sigma$ and $n$.
The aim of covariance estimation is to recover the population covariance $\Sigma$, with estimation error most commonly measured in Frobenius norm or operator norm,
the latter of which is more relevant in some applications since, by the Davis-Kahan theorem \citep{DavKah1970}, small operator norm error implies that one can recover
the leading eigenvectors of $\Sigma$.
This is of particular interest in covariance estimation when the task at hand is principal component analysis (see Section~\ref{sec:vectorrecover}), but is also relevant in other problems when $\Sigma$ is low-rank.
For example, in the case of network analysis \citep{LevLodLev2019}, the eigenvectors of $\Sigma$ encode community structure.

Even in the modestly high-dimensional regime of $p/n \to \gamma\in(0,1)$,
estimating $\Sigma$ is more challenging.
When $\Sigma = \Id$, the spectral measure of each $W_i$ satisfies the 
Mar\v{c}enko-Pastur law with parameter $\gamma$ in the large-$n$ limit.
In fact, we have the stronger result
(see Proposition~\ref{prop:avelimit} below) that
\begin{equation*}
	\|W_i - \Id \| \to \gamma + 2\sqrt{\gamma} \quad \mathrm{a.s.}
\end{equation*}
%where $\|\BM\|$ represents the operator norm of the matrix $\BM$.
A straightforward estimator of $\Sigma$ in this setting
is the arithmetic mean of the $N$ matrices,
\begin{equation*}
	\label{eqn:samplemean}
	\Ave := \frac{\sum_{i=1}^N W_i}{N},
\end{equation*}
which can be equivalently expressed as
\begin{equation*}
	\label{eqn:samplemean2}
	\Ave = \frac{\big[ X_1, \cdots, X_N\big] \big[X_1,\cdots, 
	X_N\big]^*}{Nn}.
\end{equation*}
The arithmetic mean is a sample covariance based on a total of $T=nN$
observations in this case, since we center by the known rather than estimated observation mean, and every covariance matrix is based on the same number of observations. 
Note that in the present work we assume that the observed data are mean-$0$,
and thus there is no need to center the observations about a sample mean.
This assumption comes with minimal loss of generality,
since the centered sample covariance matrix of a collection of normal
random variables is Wishart distributed with parameter $n-1$ in place of $n$,
which has no effect on the asymptotic analyses below.

\begin{remark}
In practical applications, there are situations where pooling observations
is not appropriate,
and the arithmetic mean may be ill-suited to estimating $\Sigma$ as a result.
For example, in resting state fMRI data,
pooling observations from different subjects
at a given brain region is infeasible,
as the response signals at a particular brain location
are not time-aligned across subjects.
Nonetheless, combining sample
covariance or correlation matrices across subjects
via some other procedure
may still be valid for estimating the population
covariance or correlation matrix.
\end{remark}

Throughout this paper, $T$ will denote the total number of observations of points in $p$-dimensional space.
The regime of interest is that in which $p/T \to \Gamma$,
and we will consider $T = nN$ where $N$ is a fixed number of matrices,
and $n$ will be tending to infinity with $p$. 
It will be convenient to define $\gamma = \lim p/n$,
which satisfies $\Gamma = \gamma/N$.
%Keeping the number of observed matrices $N$ fixed and letting $p$ and $n$ grow,  for the case $\Sigma = \Id$, defining $\Gamma = \lim p/T = \gamma/N$,
In this setting (see Proposition~\ref{prop:avelimit}),
\begin{equation*}
\|\Ave - \Id\| \to
\Gamma + 2\sqrt{\Gamma}\qquad \hbox{a.s.}
\end{equation*}
That is, the arithmetic mean is not a consistent estimator of $\Sigma$, even in the simple case where $\Sigma = \Id$.

As an alternative to the arithmetic mean $\Ave$, we can consider the matrix harmonic mean
\begin{equation}
	\Har := N \bigg(\sum_{i=1}^N W_i^{-1}\bigg)^{-1},
\end{equation}
provided that $n > p$ (so that the $W_i$ are invertible almost surely).
In past work \citep{Lod19}, the first author analyzed the behavior of
the harmonic mean of a collection of independent Wishart random matrices
in the regime $p/n \to \gamma\in(0,1)$.
While the harmonic mean is also inconsistent as an estimator for $\Sigma$,
we will see below that its operator-norm bias is, under
certain conditions, smaller than that of the arithmetic mean seen above.

\section{Improved Operator Norm Error of the Harmonic Mean}
\label{sec:datasplit}

When the $W_i$ are drawn from the same underlying population,
the harmonic mean can be a better estimate of the population mean $\Sigma$
in operator norm than the arithmetic mean \citep{Lod19}.
This improvement is best understood as a data-splitting result,in which we partition a sample of $p$-dimensional observations and compute the harmonic mean of the covariance estimates computed from each part.
This is certainly counter-intuitive, but we remind the reader that our intuitions are often wrong in the high-dimensional regime.

%\begin{definition}[Data Partition]
%\label{def:partitioning}
Let $D$  %\subset \BBR^p$
be a set of $T$ points in $\BBR^p$   % \ge 2p$  i.i.d.\ random points.
and let $\cP$ be a partition of $D$ into $N \ge 2$ disjoint subsets $D_i$,
\begin{equation*}
\cP :=\{D_i\}_{i=1}^N \quad\hbox{such that}\quad D=\bigcup_{i=1}^N D_i.
\end{equation*}
Define the Wishart random matrix associated with each sujbset $D_i$ as
\begin{equation*}
W(D_i) :=\frac{1}{|D_i|} \sum_{x\in D_i} x x^*,
\end{equation*}
and define the arithmetic and harmonic means associated with $\cP$ as,
respectively,
\begin{equation*}
\Ave(\cP) := \frac{1}{N} \sum_{D_i \in \cP} W(D_i)
\text{ and }
\Har(\cP) := N\Big(\sum_{D_i \in \cP} W(D_i)^{-1}\Big)^{-1},
\end{equation*}
provided that $W(D_i)$ is invertible for all $i$.
%\end{definition}

If the sets making up the partition $\cP$ are all of the same size,
then $\Ave(\cP)$ is in fact simply the sample covariance of the vectors in $D$
and does not depend on $\cP$.
The convergence of the spectrum of $\Ave(\cP)$
is classical \citep{MarPas1967}. 
The statement as given here can be found in
\citep[Theorems 3.6--7 and Theorems 5.9--10]{BaiSil2010}.
\begin{proposition}[Mar\v{c}enko-Pastur law]
\label{prop:avelimit}
Suppose $D$ is a collection of $T$ i.i.d.\ $p$-dimensional real or complex
Gaussians with zero mean and covariance $\Id$, with $p$ and $T$
tending to infinity in such a way that $p/T \to \Gamma \in (0,1/2)$,
and let $\cP$ be a deterministic sequence of partitions of $D$ 
such that $|D_i|$ are equal for all $D_i \in \cP$. The spectral
measure of $\Ave(\cP)$ converges weakly almost surely to the measure
with density
\[
\frac{1}{2\pi\Gamma x} \sqrt{(S_+ -x)(x-S_-)}\mathbf{1}_{[S_-,S_+]}(x), 
\]
where
\begin{equation}
\label{eqn:mpdef}
S_\pm = \big(1 \pm \sqrt{\Gamma}\big)^2.
\end{equation}
Further, we have the convergence
\[
\|\Ave(\cP) - \Id\| \to \Gamma + 2 \sqrt{\Gamma} \quad\mathrm{a.s.}
\]
\end{proposition}

The following result
describes the limiting behaviour of
$\Har(\cP)$ under similar conditions.
\begin{proposition}
\label{prop:harmlimit}
Let $D$ be a collection of $T$ i.i.d.\ $p$-dimensional real or complex
Gaussians with zero mean and covariance $\Id$, with $p$ and $T$
tending to infinity in such a way that
\begin{equation*}
\left| \frac{p}{T} - \Gamma \right| \le \frac{ K }{ p^2 }
\end{equation*}
where $K > 0$ is a constant and $\Gamma \in (0,1/2)$.
Let $N \geq 2$ be fixed with $T$ divisible by $N$,
and let $\cP$ be a deterministic sequence of partitions of $D$ of
size $N\leq \lfloor\Gamma^{-1}\rfloor$ such that $|D_i|$ are equal
for all $D_i \in\cP$. The spectral measure of $\Har(\cP)$ converges
weakly almost surely to the measure with density
\[
\frac{1}{2\pi \Gamma x}\sqrt{(E_+ - x)(x-E_-)}\mathbf{1}_{[E_-,E_+]}(x),
\]
where
\begin{equation} \label{eqn:HMpmdef}
E_\pm := 1 -(N-2)\Gamma \pm 2\sqrt{\Gamma} \sqrt{1 - (N-1)\Gamma}.
\end{equation}
Further, we have the convergence
\[
\|\Har(\cP) - \Id\| \to (N-2)\Gamma + 2\sqrt{\Gamma}\sqrt{1 -(N-1)
  \Gamma} \quad\mathrm{a.s.}
\]
\end{proposition}
\begin{proof}
For the complex Gaussian case, the above result  is a restatement of
\citep[Theorem 2.1]{Lod19},
 since if the $D_1,D_2,\dots,D_N$ are disjoint and
equal sized, then $W(D_i)$ are a collection of $N$ growing Wishart
matrices of the same dimension $p$ and shared parameter $n$,
with $p/n \to \gamma = N\Gamma$. As discussed in \citep[Remark 3]{Lod19}, the extension of
this result to the real Gaussian setting requires the strong asymptotic freeness of real Wishart random matrices, which was established in recent work \citep{FSW20}.
Details are supplied in Appendix~\ref{sec:strongfree}.
\end{proof}
We remark that the case where $|D_i|$ are permitted to vary in $i$,
while still feasibly handled by the tools of the paper \citep{Lod19}, is more complicated.
The limiting spectrum of the harmonic mean in this setting depends on the roots of a high-degree polynomial, whence comparison of the harmonic and arithmetic mean requires a more subtle analysis.
In contrast, when the cells of the partition are of equal size, the limiting spectral measure of the harmonic mean is characterized by the roots of a quadratic and thus
admits an explicit solution.
Thus, for the sake of concreteness and simplicity, we will assume that $\cP$ is a partition with cells of equal size for the remainder of the paper.
With the interpretation of $\Har(\cP)$ as a mean formed by splitting $D$ into equal parts, we have the following Theorem.

\begin{theorem}
\label{thm:datasplit}
Under the assumptions of Proposition~\ref{prop:harmlimit}, the operator norm
$\|\Har(\cP) - \Id\|$ is minimized for a partition $\cP$ of size $N=2$.
Further, for such a partition,
\[
\lim_{p,T \to \infty} \|\Har(\cP) - \Id\| =
2\sqrt{\Gamma}\sqrt{1-\Gamma} < \lim_{p,T\to\infty}\|\Ave(\cP) - \Id\|
= \Gamma + 2 \sqrt{\Gamma}.
\]
\end{theorem}
\begin{proof}
The function
\[
f(x) = (x-2)\sqrt{\Gamma} + 2\sqrt{\Gamma}\sqrt{1-(x-1)\Gamma}, \qquad
x < 1 +\frac{1}{\Gamma}
\]
has derivative
\[
f'(x) = \sqrt{\Gamma} - \frac{\Gamma\sqrt{\Gamma}}{\sqrt{1 +
    (x-1)\Gamma}}
\]
which is greater than 0 whenever
\[
x > 1 + \Gamma - \frac{1}{\Gamma}.
\]
For $\Gamma \in (0, 1/2)$, this region includes the point $x=2$ so
that the minimizer of $f(x)$ on the interval $[2,\infty)$ is $2$.
\end{proof}

We note that for $N=2$,
$E_+ = 1 + 2\sqrt{\Gamma} \sqrt{1 -\Gamma} < S_+$,
so that $E_+$ is closer to $1$.
That is,
at least in the case where the true covariance matrix is the identity,
the harmonic mean is shrunk toward the true population covariance
when compared with the arithmetic mean.
The above result
suggests that given a collection $D$ of $T \geq 2p$ observations,
it is better asymptotically (as measured in operator norm error)
to estimate the covariance by splitting $D$
into two equal parts $D_1$ and $D_2$ and computing the harmonic mean of
$W(D_1)$ and $W(D_2)$ than it is to
directly compute the sample covariance matrix of $D$.
The requirement that the vectors have identity covariance is partially 
addressed by \citep[Corollary 2.1.1]{Lod19}, which we restate here.

\begin{proposition} \label{prop:harmlimit:general}
  Under the same assumptions as Proposition~\ref{prop:harmlimit}, let $N=2$
  and suppose $\Sigma$ is a positive definite matrix such that
  \[
  \limsup_{p,T\to\infty}\frac{\|\Sigma\|\|\Sigma^{-1}\|\|\Har(\cP) -
    \Id\|}{\|\Ave(\cP) - \Id\|} < 1 \text{ a.s. }
  \]
  Then
  \[
  \limsup_{p,T\to\infty}\frac{\|\sqrt{\Sigma}\Har(\cP)\sqrt{\Sigma} - \Sigma\|}
  {\|\sqrt{\Sigma}\Ave(\cP)\sqrt{\Sigma} - \Sigma\|} < 1 \text{ a.s. }
  \]
\end{proposition}
%The above bound follows from a simple sub-multiplicativity argument,
%and we make no claims of optimality.
%This bound does, however, yield
%an immediate bound on general covariances.
Since multiplying $\Har(\cP)$ by $\sqrt{\Sigma}$ on both sides gives
a Wishart model with population covariance $\Sigma$
(see Remark~\ref{rem:multspik} below),
the bound in Proposition~\ref{prop:harmlimit:general}
holds so long as the condition number of $\Sigma$ lies in a certain range.
%1/2 = \gamma = 2 \Gamma, so \Gamma = 1/4.
For example, when $\Gamma = 1/4$, Proposition~\ref{prop:harmlimit:general} requires that the condition number be (asymptotically) smaller than $5/2\sqrt{3} \approx 1.44$ (see Remark 2 in \cite{Lod19} for further details).
Thus, in a certain sense, Proposition~\ref{prop:harmlimit:general} applies in the setting that is the opposite of most results on the spiked covariance model.
Namely, the harmonic mean is best suited to the case where the signal is spread over many eigenvectors, with the extreme case of this being the setting where $\Sigma = \Id$.

This leaves open the question of whether or not it is reasonable to assume that such a bound holds, given real data.
One heuristic for checking this assumption is to simply examine the spectra of $\Har(\cP)$ and $\Ave(\cP)$, but of course this runs into a circular problem, in which one must appeal to concentration of eigenvalues in order to motivate a result on operator-norm concentration.
Ultimately, the decision as to whether or not the condition number bound required by Proposition~\ref{prop:harmlimit:general} is reasonable lies with the practitioner.
%Many methods are available for both settings in the large covariance matrices literature, which covered both sparse matrices (a lot of zero entries and a fairly flat spectrum, like the identity) and the low-rank matrices (a spiked spectrum and typically not many zero entries, at least in practice). 
Nonetheless, an appealing approach would be to develop a method for estimating the population condition number rather than the full spectrum.
This would in turn give a good indication of which of the arithmetic or harmonic matrix means are better suited to the observed data.
We leave the further exploration of this line to future work.

Beyond the above operator norm results, the harmonic mean has an interesting additional property with respect to the Frobenius norm.
The arithmetic matrix mean is usually motivated as minimizing the squared Frobenius norm error.
A similar objective motivates many existing shrinkage estimators for the covariance matrix \citep{LedWol2004,FisSun2011}.
Under the setting considered above, the harmonic matrix mean, despite not being optimized for this loss, {\em matches} the Frobenius norm error of the arithmetic mean asymptotically.

\begin{lemma}
\label{eqn:lpconv}
Under the conditions of Proposition~\ref{prop:harmlimit}, when $N=2$
we have
\[
\lim_{p,T\to\infty}\frac{1}{p}\|\Har(\cP) - \Id\|_F^{2} =
\lim_{p,T\to\infty}\frac{1}{p}\|\Ave(\cP) - \Id\|_F^{2} =\Gamma \ \text{a.s.} 
\]
\end{lemma}
\begin{proof}
Since $\Har(\cP) -\Id$ is symmetric,
by the almost sure weak convergence of $\Har(\cP)$,
it suffices to show
\[
\lim_{p,T\to\infty}\frac{1}{p}\Tr\big[(\Har(\cP) - \Id)^{2}\big]
\to \int_{E_-}^{E_+}(x-1)^{2}\frac{\sqrt{(E_+
    - x)(x - E_-)}}{2\pi \Gamma x}\diff x,
\]
where $E_{\pm}$ are defined in Equation~\eqref{eqn:HMpmdef},
and compare with
\[
\lim_{p,T\to\infty}\frac{1}{p}\Tr\big[(\Ave(\cP) - \Id)^{2}\big]
\to \int_{S_-}^{S_+}(x-1)^{2}\frac{\sqrt{(S_+- x)(x - S_-)}}{2\pi \Gamma x}
\diff x,
\]
where $S_{\pm}$ are defined in Equation~\eqref{eqn:mpdef}.
Note that
\[
\frac{E_+ - E_-}{E_++E_-} = 2\sqrt{\Gamma}\sqrt{1- \Gamma},\qquad
\frac{E_+ + E_-}{2} = 1
\]
and
\[
\frac{S_+-S_-}{S_++S_-}=\frac{2\sqrt{\Gamma}}{1+\Gamma},
  \qquad \frac{S_++S_-}{2} = 1 + \Gamma.
\]
Using Lemma~\ref{lem:integral} in the Appendix,
\[
\int_{E_-}^{E_+} (x^2 - 2x + 1)\frac{ \sqrt{(E_+-x)(x-E_-)}}{2\pi
  \Gamma x}\diff x = 1-\Gamma -2(1-\Gamma) + 1
=\Gamma,
\]
while
\[
\int_{S_-}^{S_+} (x^2 - 2x + 1)\frac{ \sqrt{(S_+-x)(x-S_-)}}{2\pi \Gamma x}\diff x
=1+\Gamma -2 + 1= \Gamma.\qedhere
\]
\end{proof}

\iffalse
We pause to address the fact that Definition~\ref{def:datamodel}
requires that the random vectors be complex-valued Gaussians.
This is not unusual: many applications of random matrix theory to statistics
have first been established for complex random matrices, for technical reasons.
For example, the celebrated Ben Arous, Baik and Pech\'{e} %https://cims.nyu.edu/~benarous/
transition for spiked covariance models was first established for
complex Gaussian random vectors \citep{BBP05} and the Tracy-Widom Law
for the largest eigenvalue of a general sample covariance was first
established for complex Gaussians \citep{Ka05}.
Nonetheless, one would like to extend the above results to the real case.
The need for complex Gaussian matrices in the preceding results
arises in the proof of Theorem 2.1 in \citep{Lod19},
which can be adapted straightforwardly to real-valued data
with sub-Gaussian entries \citep[Remark 3]{Lod19}
but for its use of a result from free probability \citep{CDM07}
that is currently known only for the complex case.
We believe these results can be made to apply to real random variables,
and empirical simulations with multivariate real-valued Gaussians
in Section~\ref{sec:expts} support this claim.
\fi

\section{Rao-Blackwell Improvement of the Harmonic Mean}
\label{sec:raoblackwell}

The results in Section~\ref{sec:datasplit} are somewhat unexpected,
and raise the question of whether other matrix means have similar properties.
Analyzing other matrix means such as the geometric mean or more complicated
Frech\'{e}t means \citep{Bhatia2006,Schwartzman2016} under the high-dimensional regime poses a significant challenge since these operations currently fall outside the scope of known results in free probability theory.
Random matrix techniques do, however, allow us to extend our analysis
of the harmonic mean by computing its expectation conditioned on the
arithmetic mean.
By the Rao-Blackwell Theorem, using this conditional expectation as an estimator yields
an expected spectral norm error no worse than the unconditioned harmonic mean.

In this section we restrict ourselves to the model of Definition~\ref{def:realdatmodel}, so as to ensure the availability of explicit integrals for our quantities of interest.
We expect the same results can be established for the complex model in Definition~\ref{def:datamodel}, but doing so would require reworking the results of \citep{Kon88} (restated in the Appendix for ease of reference) for the complex Wishart ensemble, which is outside the scope of the present article.
Further, we restrict our attention to $T = 2n \ge p$ to facilitate comparison with the $N=2$ case studied in the previous section.
To this end, let $\cP = D_1 \cup D_2$ where $D_1$ and $D_2$ are disjoint, and 
\begin{align*}
W_1 &:= W(D_1)= \frac{1}{|D_1|}\sum_{x \in D_1} xx^*  = \frac{1}{n}\sum_{x\in D_1} xx^*,\\
W_2 &:= W(D_2) = \frac{1}{|D_2|} \sum_{x \in D_2} xx^* = \frac{1}{n}\sum_{x\in D_2} xx^*.
\end{align*}
The matrices $W_1$ and $W_2$ have densities \citep[Theorem
  7.2.2]{Anderson}
\begin{equation*} %\label{eqn:wishdens}
%\begin{split}
f_{W_i}(w_i) =C_{n,p}\det(w_i)^{\frac{1}{2}(n-p-1)}\exp\bigg(-\frac{n}{2}\Tr
\Sigma^{-1}w_i\bigg),
\end{equation*}
where 
\begin{equation*}
C_{n,p} =\frac{n^{\frac{pn}{2}}}{2^{\frac{1}{2}np}
\det(\Sigma)^{\frac{n}{2}}\Gamma_p\big(\frac{n}{2}\big)}, \  \Gamma_p(x) = \pi^{\frac{p(p-1)}{4}}\prod_{i=1}^p\Gamma\bigg(x
-\frac{i-1}{2}\bigg).
\end{equation*}
These densities are supported on the space $\cS_p(\BBR)$
of $p\times p$ symmetric positive definite real matrices.   As before, define
\begin{equation*} \begin{aligned}
\Ave &= \frac{W_1 + W_2}{2}, \\ %\label{eqn:avedef}
\Har &= 2\big(W_1^{-1}+W_2^{-1}\big)^{-1} %\label{eqn:hardef}
\end{aligned} \end{equation*}
and note that
\[
\Ave = \frac{1}{2n}\sum_{x \in D} x x^*
\]
is a Wishart random matrix with parameter $\Sigma$ and $2n$, and thus
\begin{equation}
\label{eqn:arithdens}
f_{\Ave}(a) := C_{2n,p}\det(a)^{\frac{1}{2}(2n-p-1)}\exp(-n\Tr
\Sigma^{-1}a),
\end{equation}
where $a$ takes values in all of $\cS_p(\BBR)$.

Recall that the matrix $\Ave$ is a sufficient statistic for the
covariance matrix $\Sigma$, and note that the loss function
\[
\ell(M,\Sigma) := \|M -\Sigma\|,
\]
is convex in the variable $M$. By
the Rao-Blackwell Theorem \cite[5a.2 (ii)]{Rao73}, we have 
\[
\BBE ~\ell\big(\mathbb{E}[\Har | \Ave], \Sigma \big)
\leq \BBE ~\ell\big(\Har,\Sigma\big) ,
\]
which is to say that as an estimator, $\Har$ is outperformed by the conditional
expectation $\BBE[\Har|\Ave]$, which we now compute.

Observe that the harmonic mean satisfies \cite[Section 4.1]{Bhatia2006}
\[
\begin{split}
  \Har &= 2W_1 - W_1 \Ave^{-1} W_1,\\
  \Har &= 2W_2 - W_2 \Ave^{-1} W_2. 
\end{split}
\]
Averaging these two equations gives 
\[
\Har = 2\Ave - \frac{1}{2} W_1 \Ave^{-1}W_1 -
\frac{1}{2}W_2\Ave^{-1}W_2,
\]
and taking the conditional expectation yields
\[
\BBE[\Har|\Ave] = 2\Ave -
\frac{1}{2}\BBE[W_1\Ave^{-1}W_1|\Ave] -
\frac{1}{2}\BBE[W_2 \Ave^{-1}W_2|\Ave].
\]
To compute the matrix-valued integrals
\[
\BBE[W_1 \Ave^{-1}W_1|\Ave] \quad\mathrm{and}\quad
\BBE[W_2\Ave^{-1}W_2 | \Ave],
\]
we proceed by directly computing the conditional density of $W_i$
given $\Ave$.

We begin with the joint density of $W_1$ and $W_2$:
\begin{equation*}
f_{W_1}(w_1) f_{W_2}(w_2) = C_{n,p}^2
\det(w_1)^{\frac{1}{2}(n-p-1)}\det(w_2)^{\frac{1}{2}(n-p-1)}
\exp\bigg[-\frac{n}{2} \Tr \big\{\Sigma^{-1}(w_1 + w_2)\big\} \bigg].
\end{equation*}
We will use this formula to obtain an expression for the joint
density of $W_1$ and $\Ave$.
For a symmetric matrix $M$ with entries $m_{i,j}$,
let $\diff m_{i,j}$ denote Lebesgue measure over that entry and define
\[
(\diff M) := \bigwedge_{1 \leq i \leq j \leq p} \diff m_{i,j},
\]
that is $(\diff M)$ is the volume form of the matrix $M$. The
``shear'' transformation
\begin{equation*}
(w_1, w_2) \mapsto (w_1, a),~~~a:= \frac{w_1 + w_2}{2},
\end{equation*}
maps the domain $\cS_p(\BBR) \times \cS_p(\BBR)$ to the region
\begin{equation*}
\{M \in \cS_p(\BBR): 0 \preceq M \preceq 2a\} \times \cS_p(\BBR),
\end{equation*}
where we remind the reader that $\preceq$
denotes the positive semidefinite ordering.
The Jacobian of this mapping is
\begin{equation*}
(\diff w_1) \wedge (\diff w_2) = 2^{\frac{p(p+1)}{2}} (\diff w_1) \wedge (\diff a) , 
\end{equation*}
hence the joint density is
\begin{equation*}
f_{W_1,\Ave}(w_1,a) =
C^2_{n,p} 2^{\frac{p(p+1)}{2}}
\det(w_1)^{\frac{1}{2}(n-p-1)}\det(2a - w_1)^{\frac{1}{2}(n-p-1)}
\exp[-n\Tr\Sigma^{-1}a].
\end{equation*}
To obtain the conditional distribution, we divide by
\eqref{eqn:arithdens}, yielding
\begin{equation}
  \label{eqn:matrixbeta}
f_{W_1|\Ave}(w_1|a) = \frac{C^2_{n,p}}{C_{2n,p}}
\frac{\det(w_1)^{\frac{n-p-1}{2}}\det(2a-w_1)^{\frac{n-p-1}{2}}}{\det(2a)^{n-\frac{p+1}{2}}},
\end{equation}
where $w_1$ is supported on the region
\[
\cD(a) := \{m \in \cS_p(\BBR) : 0 \preceq m \preceq 2a\}.
\]
Evaluating this density at $a = \Ave$, for $w_1 \in \cD(\Ave)$ yields
\begin{equation*}
f_{W_1|\Ave}(w_1|\Ave) := \frac{C_{n,p}^2}{C_{2n,p}}
\frac{\det(w_1)^{\frac{n-p-1}{2}}\det(2\Ave-w_1)^{\frac{n-p-1}{2}}}{\det(2\Ave)^{\frac{2n-p-1}{2}}},
\end{equation*}
a multivariate Beta distribution $B(p;n,n;2\Ave)$ (see Definition~\ref{def:multbeta} in the Appendix).
With this notation, we have
\begin{equation*}
\BBE[W_1 \Ave^{-1} W_1| \Ave] = \int_{\cD(\Ave)} w_1 \Ave^{-1} w_1\, f_{W_1|\Ave}(w_1|\Ave) (\diff w_1)
\end{equation*}
the integration over $w_1$ can be done using Theorem~\ref{thm:expofsquare}
in the Appendix, with $n_1=n$, $n_2=n$,
and setting $\Delta = 2\Ave$ yields the following Lemma.

\begin{lemma}\label{lem:condexp}
For any $F$ that is a function of $\Ave$ taking value in the space of
$p\times p$ matrices,
\begin{equation*}
\BBE[W_1FW_1|\Ave] = \frac{\{n(2n+1) - 2\}\Ave F \Ave + n\{(\Ave F\Ave)^\top
  + \Tr(\Ave F)\Ave\}}{(2n-1)(n+1)}.
\end{equation*}
\end{lemma}

Setting $F=\Ave^{-1}$ yields
\begin{equation*}
\BBE[W_1\Ave^{-1}W_1|\Ave] = \frac{2n(n+1) - 2 + pn}{(2n-1)(n+1)} \Ave.
\end{equation*}
The same calculation can be carried out for $\BBE[W_2 \Ave^{-1}W_2|\Ave]$
to give
\begin{equation*}
\BBE[\Har|\Ave]
= 2\Ave - \bigg\{ \frac{2n(n+1) - 2 + pn}{(2n-1)(n+1)}\bigg\}\Ave
=\frac{n(2n-p)}{(2n-1)(n+1)} \Ave,
\end{equation*}
which is simply a rescaling of $\Ave$ by a deterministic
constant. We summarize this result as a theorem.
\begin{theorem} \label{thm:raoblackhar}
  Let $T=2n$ and $D$ be as 
  in Definition~\ref{def:realdatmodel}. If $\cP$ is a partition of size $2$
  with $|D_1| = |D_2| = n$, then 
  \[
  \BBE[\Har(\cP)|\Ave(\cP)] = \frac{n(2n-p)}{(2n-1)(n+1)}\Ave(\cP).
  \]
\end{theorem}

Note that as $p/T=p/(2n) \to \Gamma \in(0,1/2)$,
the limiting spectral measure 
of the above conditional expectation converges to
%has spectral measure converging to
\[
( 1 -\Gamma) Z,
\]
where $Z$ is a random variable distributed according to the
limiting spectral distribution of $\Ave$.

The above calculations can be further extended
by making a few adjustments to the matrices $W_1$, $W_2$.
A number of matrix estimators take the form
\[
\tilde\Ave := c(\Ave + d \hat \Lambda),
\]
where $c, d$ are positive scalars and $\hat{\Lambda}$ is a positive
semidefinite matrix, all depending only on $\Ave$.
Estimators of this form have been extensively studied
in the covariance estimation literature %as a solution to the instability associated with the arithmetic mean in the high-dimensional regime
\citep{LedWol2004,FisSun2011,KI14}. 
One could take the extra step of applying the same regularization procedure to the matrices $W_1$ and $W_2$
before computing their (Rao-Blackwellized) harmonic mean.
Suppose we replace $W_1$ and $W_2$ with 
\begin{equation*}
\tilde W_1 := c(W_1 + d\hat\Lambda) \quad\hbox{and}\quad \tilde W_2 :=
c(W_2 + d\hat\Lambda),
\end{equation*}
respectively.
Letting $\tilde\Har$ be the harmonic mean of
$\tilde W_1$ and $\tilde W_2$, we can compute a Rao-Blackwell
improvement of $\tilde \Har$ in much the same way that we did for
$\Har$ above. Indeed, we still have
\begin{equation*}
\tilde \Har = 2 \tilde \Ave - \frac{\tilde W_1 \tilde \Ave^{-1} \tilde
W_1}{2} - \frac{\tilde W_2 \tilde \Ave^{-1} \tilde W_2}{2}.
\end{equation*}
We can compute the conditional expectation with respect to $\Ave$ as
follows
\begin{equation} \label{eqn:condexp1} \begin{aligned}
\BBE[\tilde \Har |\Ave] = 2 \tilde \Ave &- \frac{c^2}{2} \big(\BBE[W_1
\tilde \Ave^{-1} W_1|\Ave] + \BBE[W_2 \tilde \Ave^{-1} W_2|\Ave]\big) \\
&-c^2d \big( \Ave \tilde \Ave^{-1} \hat\Lambda +
\hat\Lambda\tilde \Ave^{-1} \Ave \big) -(cd)^2\hat\Lambda\tilde
\Ave^{-1}\hat\Lambda.
\end{aligned} \end{equation}
Using Lemma~\ref{lem:condexp}, we have 
\begin{equation} \label{eqn:condexp2}
\begin{aligned}
\frac{c^2}{2}\big(\BBE[W_1 &\tilde \Ave^{-1} W_1 | \Ave] + \BBE[W_2
  \tilde \Ave^{-1}W_2|\Ave]\big) \\
&= c^2\bigg[\frac{\{2n(n+1)-
    2\}\Ave\tilde\Ave^{-1}\Ave + n \Tr(\Ave \tilde
    \Ave^{-1})\Ave}{(2n-1)(n+1)}\bigg].
\end{aligned}
\end{equation}
Combining Equations~\eqref{eqn:condexp1} and~\eqref{eqn:condexp2},
we have
\begin{equation*} \begin{aligned}
\BBE[\tilde \Har |\Ave] = 2 \tilde \Ave
&- c^2\bigg[\frac{\{2n(n+1)-
  2\}\Ave\tilde\Ave^{-1}\Ave + n \Tr(\Ave \tilde \Ave^{-1})\Ave }{(2n-1)(n+1)}\bigg] \\
&-c^2d \big( \Ave \tilde \Ave^{-1} \hat\Lambda +
\hat\Lambda\tilde \Ave^{-1} \Ave \big) -(cd)^2\hat\Lambda\tilde
\Ave^{-1}\hat\Lambda,
\end{aligned} \end{equation*}
which we can write solely in terms of $\tilde\Ave$ and
$\hat\Lambda$ by substituting $c\Ave$ with $\tilde\Ave - cd\hat\Lambda$,
obtaining 
\begin{equation*} \begin{aligned}
\BBE[\tilde \Har |\Ave] = 2 \tilde \Ave
	&-\frac{[2n(n+1)-2]
 	[ \tilde\Ave -2cd
  	\hat\Lambda+(cd\hat\Lambda)\tilde\Ave^{-1}(cd\hat\Lambda) ] }
	{(2n-1)(n+1)} \\
&-\frac{[ np -
  ncd\Tr(\hat\Lambda\tilde\Ave^{-1}) ](\tilde\Ave-cd\hat\Lambda)}{(2n-1)(n+1)}
	- 2 cd\hat\Lambda +(cd\hat\Lambda)\tilde\Ave^{-1}(cd\hat\Lambda).
\end{aligned} \end{equation*}
We summarize the above results in the following Theorem. 
\begin{theorem}
\label{thm:raoblackreghar}
The Rao-Blackwell improvement of  $\tilde \Har$, the harmonic mean of the regularized matrices $c(W_1 + d \hat \Lambda)$  and $c(W_2 + d \hat{\Lambda})$,
where $c$ and $d$ are positive constants that only depend
on $\Ave$ and $\hat\Lambda$ is a positive definite matrix that
only depends on $\Ave$ given by 
\begin{equation*} \begin{aligned}
\BBE[\tilde \Har| \Ave] &=
n\bigg[\frac{2n -p +cd\Tr(\hat\Lambda\tilde\Ave^{-1})}{(2n-1)(n+1)}\bigg]
	\tilde\Ave
+\bigg[\frac{-n+1}{(2n-1)(n+1)}\bigg](cd\hat\Lambda)\tilde \Ave^{-1}
	(cd\hat\Lambda) \\
&~~~~~~+\bigg[\frac{2n+np-2-ncd\Tr(\hat\Lambda
  	\tilde\Ave^{-1})}{(2n-1)(n+1)}\bigg] cd\hat\Lambda.
\end{aligned} \end{equation*}
\end{theorem}

\begin{remark}
For linear shrinkage estimators of the form  
\begin{equation*}
\tilde \Ave = (1-\lambda)\Ave + \lambda \Id,
\end{equation*}
as in \citep{LedWol2004},
setting $\hat\Lambda = \Id$, $c=(1-\lambda)$, and $cd=\lambda$ in our
formula gives
\begin{equation} \label{eqn:fishsunupdate}
\begin{aligned}
\BBE[\tilde \Har| \Ave] &= n\bigg[\frac{2n -p
    +\lambda\Tr(\tilde\Ave^{-1})\}}{(2n-1)(n+1)}\bigg] \tilde\Ave
   +\bigg[\frac{-n+1}{(2n-1)(n+1)}\bigg] \lambda^2\tilde \Ave^{-1}\\
&~~~~~~+\bigg[\frac{2n+np-2-n\lambda\Tr(\tilde\Ave^{-1})}{(2n-1)(n+1)}\bigg]\lambda\Id.
\end{aligned} \end{equation}
\end{remark}

The results outlined above are unexpected, and somewhat odd.
The implication of
Theorem~\ref{thm:raoblackhar} is that the Rao-Blackwellization
of $\Har(\cP)$ is a deterministic constant multiple of $\Ave(\cP)$.
This suggests that the expense of computing $\Har(\cP)$ is not warranted,
since while $\Har(\cP)$ may improve upon $\Ave(\cP)$ as an estimator,
a scalar multiple of $\Ave(\cP)$ improves still further upon $\Har(\cP)$.
Figure~\ref{fig:4methods}
in Section~\ref{sec:expts} explores this point empirically
in the finite-sample regime.
On the other hand, the form of the Rao-Blackwellized estimator in
Equation~\eqref{eqn:fishsunupdate},
obtained from Theorem~\ref{thm:raoblackreghar}, bears noting.
In contrast to the Rao-Blackwellized version of $\Har$
considered in Theorem~\ref{thm:raoblackhar},
this estimator involves a linear combination of $\tilde \Ave$,
$\tilde \Ave^{-1}$ and $\Id$.
As a result, this estimator differs from the
linear shrinkage estimators considered elsewhere in the literature
\cite{LedWol2004,FisSun2011,Touloumis2015}.
The estimator in Theorem~\ref{thm:raoblackreghar},
which has not been proposed previously to the best of our knowledge,
falls under the heading of orthogonally invariant estimators,
as discussed in \cite{DonGavJoh2018}.
Thus, in particular, there should exist some orthogonally invariant loss
function for which the estimator in Equation~\eqref{eqn:fishsunupdate}
is the optimal estimator.
We leave further study of this estimator and its properties to future work.

\section{Eigenvector Recovery}
\label{sec:vectorrecover}

A major motivation for working with the operator norm is to obtain
guarantees on convergence of eigenvectors,
which are often the main object of interest in covariance estimation,
as in when the covariance is used for principal component analysis.
This is done via the Davis-Kahan theorem,
which bounds the distance between the leading eigenvectors $v_1(\hat\Sigma)$
and  $v_1(\Sigma)$ in terms of $\|\hat\Sigma  -\Sigma\|$.
For example, it can be shown \citep[Corollary 1]{YuWanSam2015} that
\begin{equation}
\label{eqn:DavisKahan}
\|v_1(\hat\Sigma) - v_1(\Sigma)\| \leq \frac{2^{\frac{3}{2}}\|\hat\Sigma - \Sigma\|}{\lambda_1(\Sigma) - \lambda_2(\Sigma)} \quad\hbox{if}\quad \langle v_1(\hat\Sigma),v_1(\Sigma)\rangle > 0
\end{equation}
In this section, we show that under a spiked covariance model,
the leading eigenvector of $\Har(\cP)$ carries information about the
leading eigenvector of the population covariance matrix $\Sigma$
in the regime $p/T \to \Gamma \in(0,1/2)$,
and we compare its performance with that of the leading eigenvector of $\Ave(\cP)$. 

\begin{definition}
\label{def:spikedmodel}
Let $D$ be a set of $T$
i.i.d.\ $p$-dimensional centered multivariate  real or complex Gaussians with
population covariance matrix
\begin{equation} \label{eqn:spikedmodel}
\Sigma = \Id + \theta vv^*,
\end{equation}
where $\theta > 0$ and $v$ is a $p$-dimensional (real or complex) unit-norm vector.
As in Proposition~\ref{prop:harmlimit}
assume that
\begin{equation*}
\Big| \frac{p}{T} - \Gamma\Big| \leq \frac{K}{p^2},
\end{equation*}
where $\Gamma \in (0,1/2)$ and $K > 0$ is a constant that does not depend on $p$ or $T$.
\end{definition}
\begin{remark}
\label{rem:multspik}
Let $D^0$ be a collection of multivariate real or complex Gaussians
with zero mean and covariance $\Id$.
If we define
\[
D = \{ \sqrt{\Sigma}x : x \in D^0\}=: \sqrt{\Sigma}D^0,
\]
where $\Sigma$ is given in Definition~\ref{def:spikedmodel}, then $D$
has the same distribution as the model in
Equation~\ref{eqn:spikedmodel}.
Moreover, by this same
transformation, we may take a partition $\cP^0$ of $D^0$
and generate a partition $\cP$ of
$D$ by replacing each $D_i^0$ in $\cP^0$ by $D_i:=\sqrt{\Sigma}D_i^0$.
With this definition we have the equality
\[
\Har(\cP) = \sqrt{\Sigma}\Har(\cP^0)\sqrt{\Sigma} \quad\hbox{and}\quad
\Ave(\cP) = \sqrt{\Sigma}\Ave(\cP^0)\sqrt{\Sigma}.
\]
\end{remark}

The Theorem below follows from well-established results in the
literature and can be generalized without any change to higher-rank
perturbations of the identity.
We focus here on the simple case of one spike in order to get more explicit insight into the performance of $\Har(\cP)$. 

\begin{theorem}
\label{thm:spikerecovery}
Let $D$ be a spiked model as in Definition~\ref{def:spikedmodel}, and 
suppose $\cP$ is a partition satisfying the conditions of
Proposition~\ref{prop:harmlimit}, with $N=2$. Then we have the almost sure
convergence
\[
\lambda_1\big\{\Har(\cP)\big\}\to\begin{cases} 1+\frac{\Gamma}{\theta} +
(1-\Gamma)\theta & \hbox{if $\theta >\sqrt{\frac{\Gamma}{1-\Gamma}}$}\\
1+ 2\sqrt{\Gamma}\sqrt{1-\Gamma} & \hbox{otherwise}, \end{cases}
\]
and
\[
\Big|\big\langle v_1\big\{\Har(\cP)\big\}, v\big\rangle \Big|^2 \to 
\begin{cases}
\frac{\theta + 1}{\theta}\frac{\theta^2(1-\Gamma) - \Gamma}{\theta^2(1-\Gamma) +\theta + \Gamma}&
\hbox{if $\theta > \sqrt{\frac{\Gamma}{1-\Gamma}}$}, \\ 0 & \hbox{otherwise.}
\end{cases}
\]
\end{theorem}

\begin{proof}
To prove this result we will use the general framework  
\citep{BGR11},
 which considers multiplicative spikes of the form
\begin{equation*}
\tilde{M} := \sqrt{\Id + \theta vv^*} M\sqrt{\Id + \theta vv^*}.  
\end{equation*}
Here $\theta$ and $v$ are as in Definition~\ref{def:spikedmodel} and
$M$ is a symmetric (or Hermitian if $M$ has complex entries) matrix whose eigenvalue 
distribution converges weakly to a spectral measure $\nu$ almost surely, and $\nu$ is supported on the interval $[a,b]$. 
Assume further that the convergence of the largest (smallest) eigenvalue of $M$
is the right (left) edge of the support of $\nu$ and that the
distribution of $M$ is invariant under conjugation by an orthogonal matrix (unitary if $M$ has complex entries).
Recall that this implies that the matrix of eigenvectors of $M$ is Haar distributed on
the orthogonal group (unitary group for the complex case).
Define for $z \in \BBC\backslash [a,b]$
\begin{align*}
m_\nu(z) &:= \int_{\BBR} \frac{\nu(\diff x)}{z-x},\\
t_\nu(z) &:= \int_{\BBR} \frac{x\nu(\diff x)}{z-x} = -1+z m_\nu(z),
\end{align*}
and let $t_\nu^{-1}(z)$ be the functional inverse of $t_\nu$.
By \citep[Theorem 2.7]{BGR11} we have the almost sure convergence
\[
\lambda_1(\tilde{M}) \to \begin{cases}
t_\nu^{-1}\Big(\frac{1}{\theta}\Big) & \theta> \frac{1}{t_\nu(b^+)},\\
b & \mathrm{otherwise}.
\end{cases}
\]
Here $t_\nu(b^+)$ is the limit as $z \to b$ of $t_\nu(z)$, well defined when $\nu$ has a density with square root decay near $b$
\citep[Proposition 2.10]{BGR11}. Furthermore, by \citep[Remark 2.11]{BGR11} we have the almost sure convergence
\[
|\langle v_1(\tilde{M}), v \rangle|^2 \to \begin{cases}
-\frac{\theta + 1}{\theta^2 \rho t_\nu'(\rho) } & \hbox{if $\theta > 1/t_\nu(b)$,}\\
0 &\hbox{otherwise},
\end{cases}
\]
where
\[
\rho := t_\nu^{-1}\bigg(\frac{1}{\theta}\bigg).
\]

Applying these results using Remark~\ref{rem:multspik}
and taking $M=\Har(\cP^0)$ where $\cP^0$ is the partition of a data set $D^0$ with population covariance $\Id$,
we see that $M$ satisfies the required
convergence properties by Proposition~\ref{prop:harmlimit} and is unitarily
invariant. Letting $\nu$ equal to the limiting spectral measure of
$\Har(\cP^0)$, and noting for $N=2$
\[
E_{\pm} = 1 \pm 2 \sqrt{\Gamma}\sqrt{1-\Gamma},
\]
the proof now proceeds by calculation.  From the results of
\citep[Equation (18)]{Lod19}, $m_\nu(z)$ satisfies the fixed point equation
\[
\Gamma z m_\nu(z)^2 + (1-2\Gamma - z)m_\nu(z) + 1 = 0 . 
\]
Inserting the definition of $t_\nu(z)$ and simplifying yields
\[
\Gamma t_\nu(z)^2 + (1-z)t_\nu(z) + 1-\Gamma = 0.
\]
Taking the limit as  $z$ goes to $E_+$ and utilizing the square
root decay of $\nu$ at $E_+$ yields
\[
\Bigg(t_\nu(E_+) - \sqrt{\frac{1-\Gamma}{\Gamma}}\Bigg)^2 = 0
\implies t_\nu(E_+) = \sqrt{\frac{1-\Gamma}{\Gamma}}.
\]
Hence, a phase transition in the largest eigenvalue of $\Har(\cP)$
occurs for
$\theta > \sqrt{\Gamma/(1-\Gamma)}$.
We can solve for the inverse of $t_\nu(z)$ by substituting
$z = t_\nu^{-1}(w)$ into the polynomial fixed point equation for
$t_\nu(z)$,
\[
t_\nu^{-1}(w) = 1 + \Gamma w + \frac{1-\Gamma}{w}.
\]
Assuming $\theta > \sqrt{\Gamma/(1-\Gamma)}$ and inserting $w=1/\theta$
gives the location of the spiked eigenvalue of $\Har(\cP)$,
\begin{equation*}
\rho = t_\nu^{-1}\bigg(\frac{1}{\theta}\bigg) = 1 + \frac{\Gamma}{\theta}+ (1-\Gamma)\theta.
\end{equation*}
Differentiating the fixed point equation of $t_\nu(z)$
gives the following fixed point equation for $t_\nu'(z)$
\begin{equation*}
(2\Gamma t_\nu(z) +1-z) t_\nu'(z) - t_\nu(z) =0. 
\end{equation*}
Substituting $\rho$ yields
\[
t_\nu'(\rho) = \frac{1}{2\Gamma + \theta(1 - \rho)} = \frac{1}{\Gamma -
\theta^2(1-\Gamma)}, 
\]
and hence
\[
-\frac{\theta + 1}{\theta^2\rho t_\nu'(\rho)} = 
\frac{\theta + 1}{\theta^2}\frac{\theta^2(1-\Gamma) - \Gamma}{1 +
\frac{\Gamma}{\theta} + (1-\Gamma)\theta}, 
\]
which concludes the proof.
\end{proof}

\begin{remark}
\label{rem:arithmetspike}
The same analysis has been performed on $\Ave(\cP)$
\citep{BBP05,De07,HoRa07,Bo08}. Under this setting, we have
the almost sure convergence \citep[Section 3.2]{BGR11}
\[
\lambda_1\big\{\Ave(\cP)\big\} \to\begin{cases}
(\theta +1)\big(1 + \frac{\Gamma}{\theta}\big) &\hbox{if $\theta > \sqrt{\Gamma}$,}\\
\big(1+ \sqrt{\Gamma}\big)^2 &\hbox{otherwise},
\end{cases}
\]
and
\[
\Big|\big\langle v_1\big\{\Ave(\cP)\big\}, v \big\rangle \Big|^2 \to
\begin{cases}
\frac{1 - \frac{\Gamma}{\theta^2}}{1 + \frac{\Gamma}{\theta}} & \hbox{if
$\theta > \sqrt{\Gamma}$,}\\
0 & \hbox{otherwise.}
\end{cases}
\]
\end{remark}

When $0 < \Gamma < \frac{1}{2}$, we have
$\sqrt{\Gamma/(1-\Gamma)} > \sqrt{\Gamma}$,
and it is possible to choose $\theta$ such that
\begin{equation} \label{eq:thetarange}
\sqrt{\Gamma}< \theta  \leq \sqrt{\frac{\Gamma}{1-\Gamma}}.
\end{equation}
Comparing the phase transition for the harmonic mean given in
Theorem~\ref{thm:spikerecovery} and that for the arithmetic mean
given in Remark~\ref{rem:arithmetspike}, we see that when
$\theta$ satisfies Equation~\eqref{eq:thetarange},
Theorem~\ref{thm:spikerecovery} and Remark~\ref{rem:arithmetspike}
imply the almost sure convergence
\begin{align*}
\Big|\big\langle v_1\big\{\Har(\cP)\big\}, v\big\rangle \Big|^2 &\to 
0 , \\
\Big|\big\langle v_1\big\{\Ave(\cP)\big\}, v \big\rangle \Big|^2 &\to
\frac{1 - \frac{\Gamma}{\theta^2}}{1+\frac{\Gamma}{\theta}}.
\end{align*}
This means that for low signal strength $\theta$,
$v_1\big\{\Har(\cP)\big\}$ fails to have any relationship with $v$.

On the other hand, when $\theta > \sqrt{\Gamma/(1-\Gamma)}$ the leading eigenvectors of both $\Har(\cP)$ and $\Ave(\cP)$ have some relationship with $v$ in the limit as $p/T \to \Gamma$. We observe that
\begin{equation*} \begin{aligned}
\lim_{p,T\to\infty} &\bigg(\Big|\big\langle v_1\big\{\Ave(\cP)\big\},v \big\rangle\Big|^2  - \Big|\big\langle v_1\big\{\Har(\cP)\big\},v \big\rangle\Big|^2   \bigg)\\
&=\frac{1 - \frac{\Gamma}{\theta^2}}{1 + \frac{\Gamma}{\theta}} - \frac{\theta + 1}{\theta}\frac{\theta^2(1-\Gamma)- \Gamma}{\theta^2(1-\Gamma) + \theta + \Gamma}
= \frac{\Gamma^2(1+\theta)^2}{\big(1 + \frac{\Gamma}{\theta}\big)\theta \big\{ \theta^2(1-\Gamma) + \theta + \Gamma\big\}} > 0,
\end{aligned} \end{equation*}
so that the leading eigenvector of $\Ave(\cP)$ functions as a better estimator,
asymptotically, for all possible values of $\theta$.

Compare this result with the bound predicted by solely
analyzing the upper bounds obtained from the Davis-Kahan Theorem.
That is, taking $\lambda_1(\Sigma) - \lambda_2(\Sigma) = \theta$
in Equation~\eqref{eqn:DavisKahan},
we have
\begin{equation*} \begin{aligned}
\Big\|v_1\big\{\Har(\cP)\big\}- v\Big\| &\leq \frac{2^{\frac{3}{2}}\|\Har(\cP) - \Sigma\|}{\theta} \quad\hbox{when}\quad \big\langle v_1\big\{\Har(\cP)\big\},v\big\rangle>0, \text{ and }\\
\Big\|v_1\big\{\Ave(\cP)\big\} - v\Big\| &\leq \frac{2^{\frac{3}{2}}\|\Ave(\cP) - \Sigma\|}{\theta}\quad\hbox{when}\quad \big\langle v_1\big\{\Ave(\cP)\big\},v\big\rangle>0.\\
\end{aligned} \end{equation*}
Since the condition number of $\Sigma$ is $1+\theta$,
by Proposition~\ref{prop:harmlimit:general} we have
\begin{equation*}
\limsup_{p,T\to\infty} \frac{\|\Har(\cP) - \Sigma\|}{\|\Ave(\cP) - \Sigma\|} < 1
\text{ whenever }
\theta < \frac{2 + \sqrt{\Gamma}}{2\sqrt{1-\Gamma}} - 1.
\end{equation*}
In order for $|\langle v_1\{\Ave(\cP)\}, v\rangle|^2 \not\to 0$,
we would need $ \theta > \sqrt{\Gamma}$. Notice that as $\Gamma \to  1/2$,
\begin{equation*}
\lim_{\Gamma \to \frac{1}{2}} \frac{2 + \sqrt{\Gamma}}{2\sqrt{1-\Gamma}} - 1
= \sqrt{2}-  \frac{1}{2} > \lim_{\Gamma \to \frac{1}{2}} \sqrt{\Gamma}
= \frac{1}{\sqrt{2}}.
\end{equation*}
It follows that there exist values of $\theta$ and $\Gamma$
for which the Davis-Kahan theorem
predicts that the harmonic mean yields
better eigenvector recovery than the arithmetic mean,
even though $v_1\{\Har(\cP)\}$
has no asymptotic relationship to $v$ while $v_1\{\Ave(\cP)\}$ does.

\section{Experiments} \label{sec:expts}

Our theoretical results presented above are asymptotic,
so we complement them here with
a brief investigation of the empirical finite-sample
performance of the harmonic and arithmetic matrix means 
and related estimators.
We begin by comparing the operator norm error of the arithmetic and
harmonic matrix means in recovering the true covariance matrix.
We generate a random covariance matrix
$\Sigma = U D U^T \in \BBR^{p \times p}$
by choosing $U \in \BBR^{p \times p}$ according to Haar measure on the
orthogonal matrices and populating the entries of the diagonal matrix $D$
with i.i.d.\ draws from the uniform distribution on the interval $[1,b]$.
The parameter $b \ge 1$ serves as a proxy for the condition number of 
$\Sigma$ (e.g., $b=1$ corresponds to $\Sigma = \Id$).
We then draw $T=4p$ independent mean 0 normal random vectors
with covariance matrix $\Sigma$.
Splitting these $T$ random vectors into two equal size samples of size $n=T/2$,
we compute the harmonic mean of the two resulting sample covariances.   
We compare the performance of this estimator to the
sample covariance matrix of the full sample
(i.e., the arithmetic mean of the two splits).
We also include, for the sake of comparison,
the shrinkage estimator proposed by Fisher and Sun \citep{FisSun2011} specifically for normal data in the high-dimensional setting, which should improve on the sample covariance matrix.  
Similar to the scheme originally proposed by Ledoit and Wolf \citep{LedWol2004}, this estimator is a convex combination of the sample covariance matrix and a target matrix,
which we take here to be the identity.

\begin{figure*}[t]
  \centering
  \includegraphics[width=\textwidth]{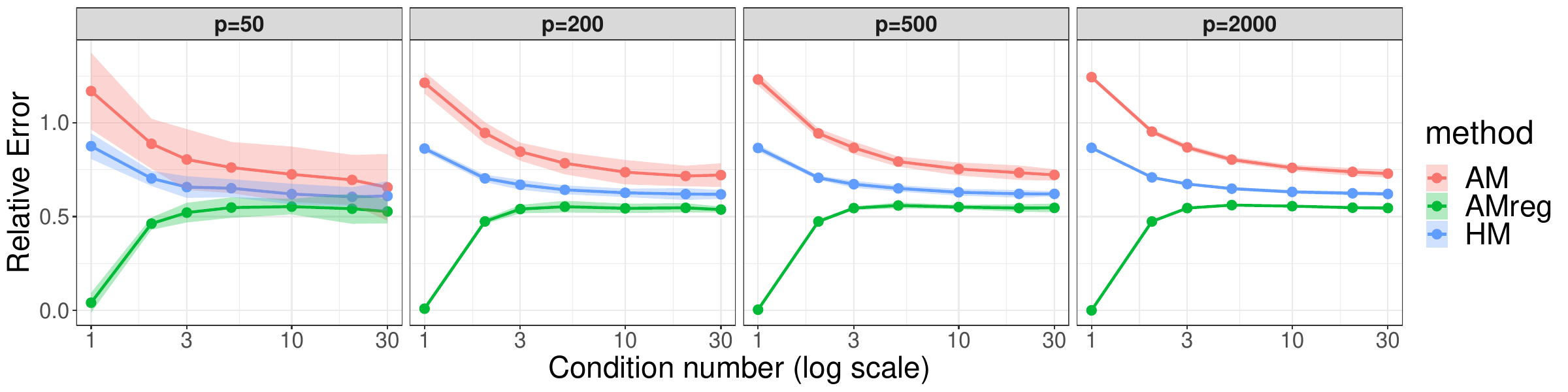}
  \vspace{-8mm}
  \caption{\small Operator norm relative error
	$\|\hat \Sigma - \Sigma \| /\| \Sigma\|$
	as a function of the condition number parameter $b$
	for different choices of the data dimension $p$.
	The plot compares the arithmetic matrix mean (red),
	the Fisher-Sun estimator (green)
	and the harmonic matrix mean (blue).
	Each point is the mean of $20$ independent trials, with the shaded
	regions indicating two standard deviations.
	}
	\label{fig:haar:pfacet}
\end{figure*}

\begin{figure*}[t]
  \centering
  \includegraphics[width=\textwidth]{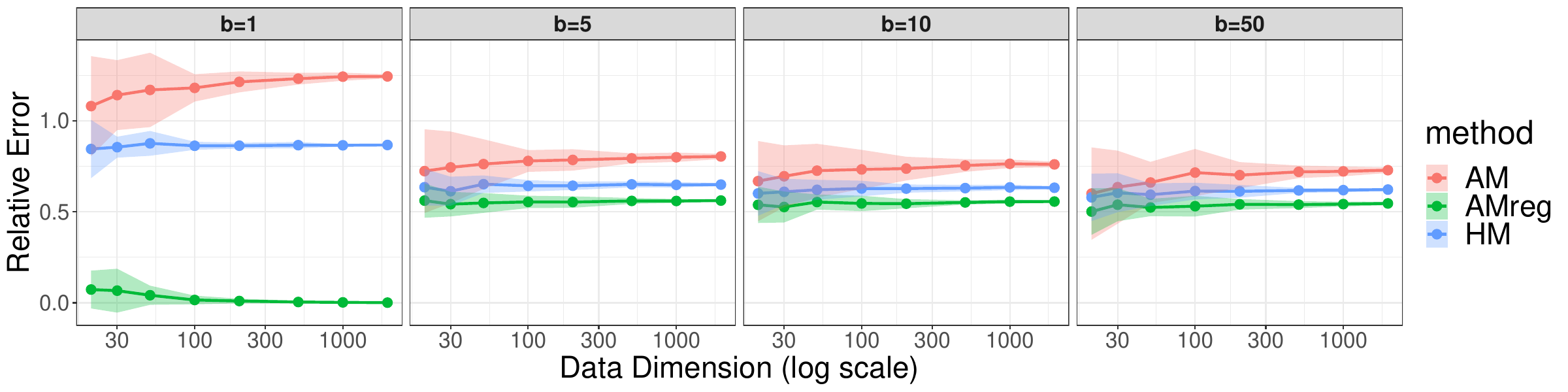}
  \vspace{-8mm}
  \caption{\small Operator norm relative error in recovering the true
    covariance matrix $\Sigma$ as a function of the data dimension
    for different choices of condition number parameter $b$
	for the arithmetic mean (red), Fisher-Sun regularized
	sample covariance (green) and harmonic mean (blue).
	Each point is the mean of $20$ independent trials, with the shaded
	regions indicating two standard deviations.
	}
	\label{fig:haar:bfacet}
\end{figure*}

Figures~\ref{fig:haar:pfacet} and~\ref{fig:haar:bfacet} 
display, for various choices of
the dimension $p$ and the condition number $b$,
the operator norm relative error 
in recovering $\Sigma$ for these three estimators.
Over a wide range of condition numbers,
the harmonic mean yields a better estimate of the population
covariance $\Sigma$ than does the arithmetic mean,
but does not manage to match the performance of the Fisher-Sun regularized estimator.  
Unsurprisingly,
when the regularization target matrix is close to the truth,
as is the case when the condition number $b$ is close to $1$,
regularization yields an especially large improvement in estimation error,
but the gap between the harmonic mean and the Fisher-Sun estimator
narrows substantially as soon as the condition number becomes even
moderately large, corresponding to the population covariance matrix
being far from the identity.
In keeping with Proposition~\ref{prop:harmlimit:general},
which suggests that we should only expect the harmonic mean to yield
improvement over the arithmetic mean for suitably small condition numbers,
the size of the improvement of harmonic mean over the (unregularized)
arithmetic mean does appear to shrink as the condition number increases.
Note that performance stabilizes as $b$ and $p$ increase
because we are assessing the estimators according to relative error,
not because the problem is necessarily becoming easier for larger values
of these parameters.

It is natural to ask how these estimators compare as the number of
observations vary.
%That is, how do the number of i.i.d.\ zero-mean normal random vectors influence the performance of  the two estimators?
Toward that end, consider the same experimental setup discussed above,
but taking $\lceil 2qp \rceil$ samples, with $q > 1$ to ensure that
splitting the observations into two samples yields two invertible sample covariances.
Larger values of $q$ correspond, roughly, to better-conditioned
sample covariance matrices.
Figure~\ref{fig:haar:nsweep} shows how the three estimators of the
population matrix mean compare as a function of this parameter $q$
for different choices of the dimension $p$ and condition number $b$.
We see that as the number of samples increases (i.e., as $q$ increases),
the improvement of the harmonic mean over the arithmetic mean
decreases.
This is in keeping with Proposition~\ref{prop:harmlimit} as well as the intuition that as the number of samples increases, the sample covariance becomes a more stable (though still not consistent) estimate of the population covariance.

In Section~\ref{sec:raoblackwell}, we saw that in the $N=2$ case,
the matrix harmonic mean $\Har$ could be further improved
by Rao-Blackwellization.
Thus, we have four possible estimates of the population covariance:
the arithmetic mean, the Fisher-Sun regularization of the arithmetic mean,
the harmonic mean and the Rao-Blackwellization of the harmonic mean.
Figure~\ref{fig:4methods} compares these four different estimators,
under the same experimental setup as the one described above.
The plot shows that Rao-Blackwellizing the harmonic mean does little
to change its behavior. Indeed, the performance of the harmonic mean
and its Rao-Blackwellization are so similar that their lines overlap
in the plot.
As in the plots above, the arithmetic mean performs more poorly as an
estimator compared to the harmonic mean, but regularization
using the method of Fisher and Sun
improves its performance over that of the harmonic mean, if only slightly.

\begin{figure*}%[t!]
  \centering
  \includegraphics[width=0.8\textwidth]{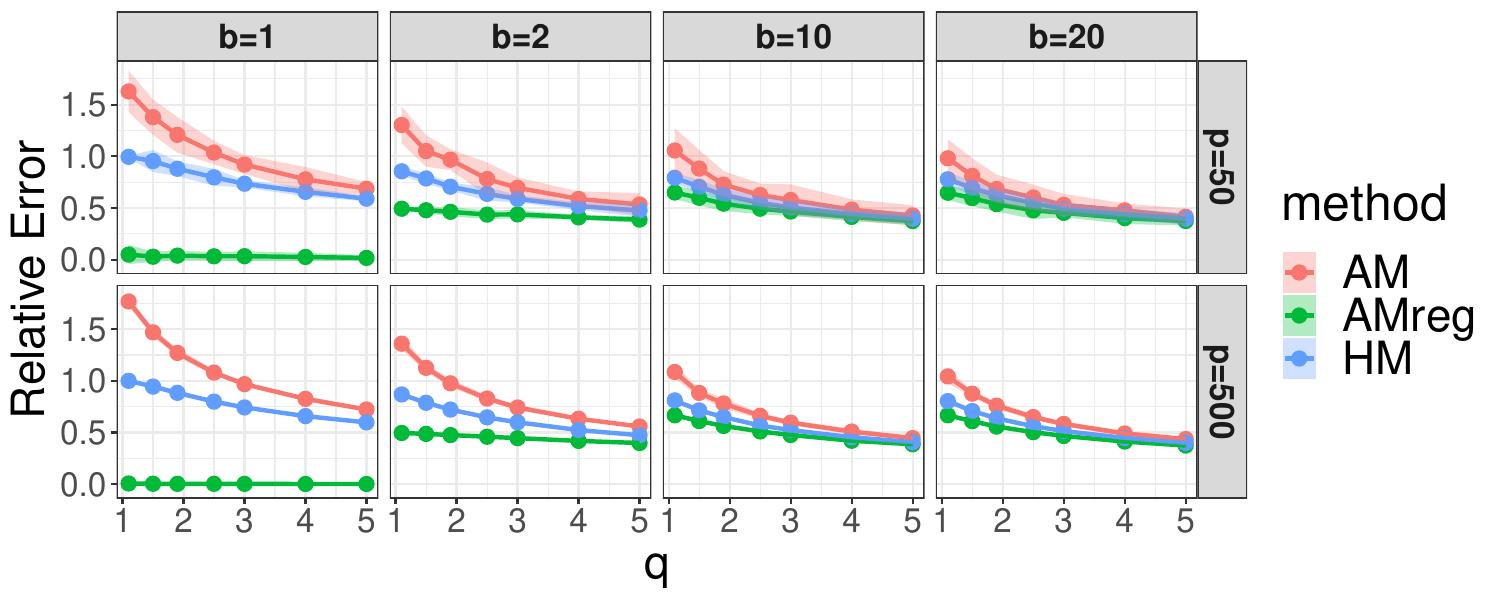}
  \vspace{-4mm}
  \caption{\small Operator norm relative error in recovering the true covariance matrix $\Sigma$ as a function of the number of normal samples,
	parameterized by $q$, for different choices of the data dimension $p$
	and the condition number $b$.
	for the arithmetic mean (red),
	Fisher-Sun regularization of the arithmetic mean (green),
	and harmonic mean (blue).
	Each point is the mean of $20$ independent trials, with the shaded
	regions indicating two standard deviations.
} \label{fig:haar:nsweep}
\end{figure*}

\begin{figure*}[h]
  \centering
  \includegraphics[width=\textwidth]{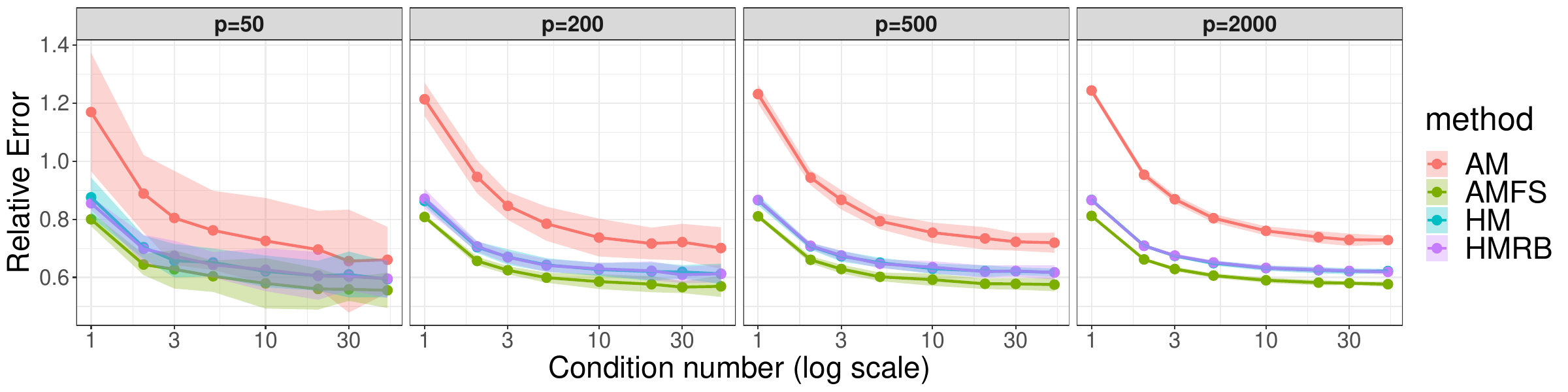}
  \vspace{-7mm}
  \caption{\small Operator norm relative error in recovering the true
    	covariance matrix $\Sigma$ as a function of the condition number $b$
    	for different choices the data dimension $p$
        for the arithmetic mean (red), 
	Fisher-Sun regularization of the arithmetic mean (green),
	harmonic mean (blue) and
	Rao-Blackwellized harmonic mean (purple).
        Each point is the mean of $20$ independent trials, with the shaded
        regions indicating two standard deviations.
	Note that the lines corresponding to the harmonic mean
	and its Rao-Blackwellization overlap.}
	\label{fig:4methods}
\end{figure*}

We close by briefly exploring the eigenvector recovery results discussed in Section~\ref{sec:vectorrecover}.
Recall that the spiked eigenvector estimation problem considered in that section concerns Wishart matrices with covariance $\Sigma = \Id + \theta v v^*$, where $\theta > 0$ and $v \in \BBR^p$ has unit norm, and the goal is to recover the spike eigenvector $v$ based on $N=2$ Wishart matrices, each constructed on $n$ independent mean-$0$ covariance-$\Sigma$ normals.
Theorem~\ref{thm:spikerecovery} and Remark~\ref{rem:arithmetspike} predict that (in the large-$p$ limit) as $\theta$ increases, the absolute value of the inner products between $v$ and the leading eigenvector of both the arithmetic and the harmonic mean increase to $1$.
Further, this behavior undergoes a phase transition-like change, the location of which is determined by $\Gamma = \lim p/Nn$.
In particular the arithmetic mean undergoes its phase transition at $\theta = \sqrt{\Gamma}$, and the harmonic mean undergoing its phase transition at $\theta = \sqrt{\Gamma/(1-\Gamma)}$.
Figure~\ref{fig:vectorrecover} examines this behavior in the finite-sample regime.

Consider a pair of independent Wishart matrices, each based on $n=4000$ independent normals of dimension $p=2000$ with mean $0$ and covariance $\Sigma = \Id + \theta v v^*$ where $\theta > 0$ and $v \in \BBR^p$ has unit norm.
That is, we are under the setting of Theorem~\ref{thm:spikerecovery}, with $N=2$ and $\Gamma=p/Nn =1/4$.
Having generated two such Wisharts, we can compute the arithmetic and harmonic means of these two covariance matrices and compare, for each of these two different means, how well the leading eigenvector $\hat v$ estimates the true spike eigenvector $v$, as measured by $|\langle \hat v, v \rangle|$.
Figure~\ref{fig:vectorrecover} summarizes this experiment.
The plot shows, for both the arithmetic (red) and the harmonic (blue) matrix means, the recovery of the spike $v$ by the leading eigenvector of the matrix mean, as a function of the signal strength $\theta > 0$.
Each data point is the mean of twenty independent trials, with error bars indicating two standard errors of the mean.
With $\Gamma=1/4$, the theory presented in Section~\ref{sec:vectorrecover} predicts that in the large-$p$ limit, below $\theta = 1/\sqrt{3} \approx 0.577$, the inner product between the leading eigenvector of the harmonic mean with $v$ should be close to zero.
Similarly, in the case of the arithmetic mean, the inner product between $v$ and the leading eigenvector of the arithmetic mean should be close to zero below $\theta = 1/2$.
Of course, Figure~\ref{fig:vectorrecover} reflects this inner product for the finite-sample setting $p=2000$.
Nonetheless, examining the plot, we see that the predictions of Section~\ref{sec:vectorrecover} are borne out.
Both the arithmetic and harmonic mean improve in their detection of the spike as $\theta$ increases, with the arithmetic mean appearing to detect the spike eigenvector before the harmonic mean does.
Also as predicted by the theory, the arithmetic mean maintains better estimation of the leading eigenvector at all values of $\theta$.

\begin{figure*}[h]
  \centering
  \includegraphics[width=0.7\textwidth]{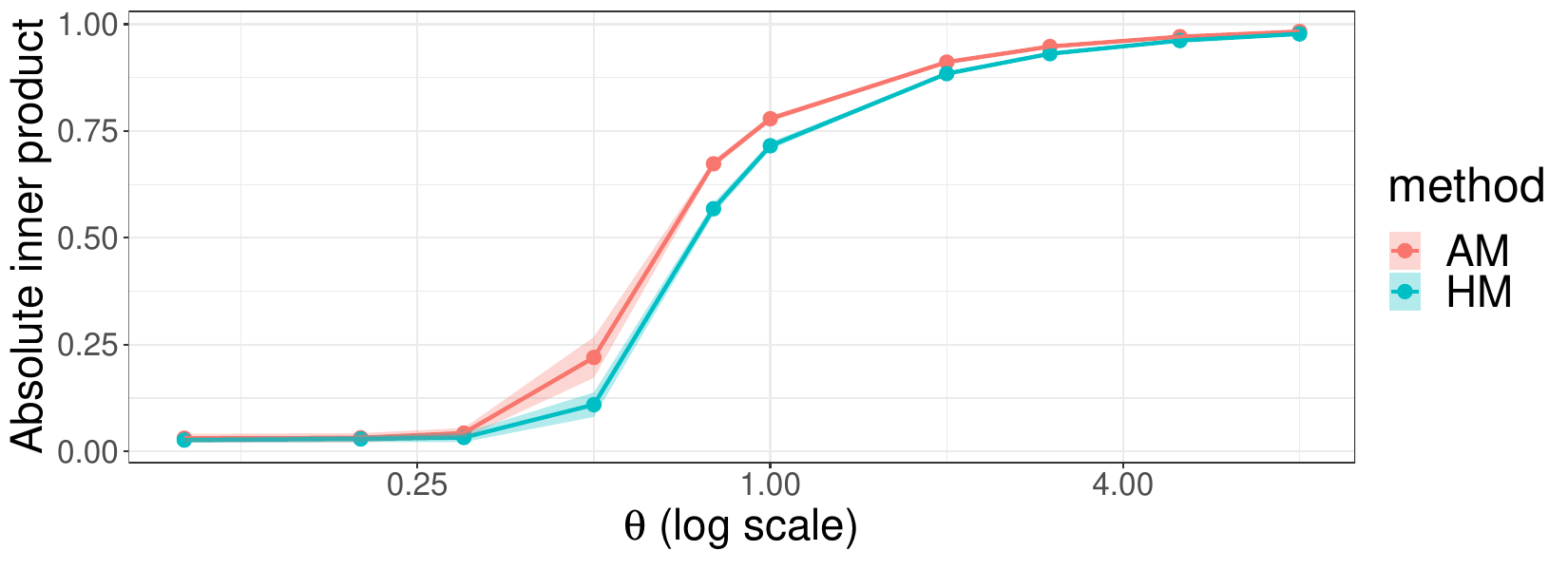}
  \vspace{-4mm}
  \caption{\small Recovery of the spike eigenvector $v$, as measured by the absolute value of the inner product between $v$ and the leading eigenvector of the arithmetic matrix mean (red) and the harmonic matrix mean (blue), as a function of the signal strength $\theta$. Each data point is the mean of twenty independent trials, with error bars indicating two standard errors of the mean.}
	\label{fig:vectorrecover}
\end{figure*}

\section{Discussion} \label{sec:discussion}

The results presented here seem to contradict our intuition about matrix
estimation under various matrix norms.
It is common in the literature to measure the quality of an estimator
$\hat \Sigma$ by the limit of the Frobenius norm error
$\|\hat\Sigma -\Sigma\|_F$, and shrinkage estimators
such as \citep{LedWol2004} are designed to minimize this quantity.
Since the operator norm is bounded above by the Frobenius norm,
controlling the Frobenius norm error is sufficient to control 
the operator norm.
However, in the high-dimensional regime, these arguments become more subtle.
One must often normalize the Frobenius norm
by the matrix dimension to ensure convergence
and obtain a sensible asymptotic analysis.
Typically, this normalization takes the form
$p^{-\frac{1}{2}}\|\hat{\Sigma} - \Sigma\|_F$.
The analogous operator norm quantity,
$p^{-\frac{1}{2}}\| \hat{\Sigma} - \Sigma \|$,
converges to zero in many common settings,
rendering the upper bound in terms of
$p^{-\frac{1}{2}}\|\hat{\Sigma} - \Sigma\|_F$ (asymptotically) trivial.
For example, when $\hat{\Sigma}$ is Wishart-distributed
(as happens when $\hat{\Sigma}$ is a sample covariance),
$\| \hat{\Sigma} - \Sigma\| = O(1)$.
Thus, obtaining non-trivial bounds on  
$\|  \hat{\Sigma} - \Sigma\|$ requires direct analysis rather than a Frobenius norm bound.

Ultimately, the choice to analyze the operator norm error
as opposed to the Frobenius norm error or some other matrix norm
is guided by the inference task at hand, and the available information about
the population covariance $\Sigma$ that one wishes to capture
in the estimator $\hat\Sigma$.  For the task of recovering the leading eigenvector(s) of $\Sigma$,   the operator norm takes a more prominent
role due to the Davis-Kahan inequality,
but here again the resulting bound need not be tight,
and the resulting bound on eigenvector recovery may be trivial.

In summary,
our results highlight the ways in which the harmonic mean $\Har(\cP)$
may outperform the arithmetic mean $\Ave(\cP)$ as an estimator
of the population covariance $\Sigma$:
\begin{itemize} 
\item Under certain conditions on $\Sigma$,
	the operator norm error of $\Har(\cP)$ is {\em better than}
	that of $\Ave(\cP)$.
	In particular, when $\Sigma=\Id$,
	the normalized Frobenius norm error of
	$\Har(\cP)$ {\em matches} that of $\Ave(\cP)$, despite
	$\Har(\cP)$ not being optimized for the Frobenius norm loss.
	We have observed similar behavior empirically
	when $\Sigma$ is close but not equal to the identity.
\item For a spiked model $\Sigma= \Id + \theta v v^*$,
	there is a range of values of $\theta$
	for which the eigenvector recovery using $\Har(\cP)$ 
	is always {\em worse than} that obtained using $\Ave(\cP)$,
	even though  $\Har(\cP)$ provides a {\em better} operator norm error
	than $\Ave(\cP)$. 
\end{itemize}
We are not aware of any other estimators with these two properties,
let alone of one that has the interpretation of being a mean with respect
to a different geometry.
Moreover, the fact that $\Har(\cP)$ can be interpreted as the result of
a data-splitting procedure suggests the possibility
of other procedures with similar interpretations that
improve over classical estimators in the high-dimensional regime.
The Rao-Blackwellization of $\Har(\cP)$ shows that,
for the purpose of covariance estimation,
a similar or better improvement over $\Ave(\cP)$
can also be achieved by a suitably-chosen scalar multiple of $\Ave(\cP)$.
While this shows that $\Har(\cP)$ has better-performing alternatives in
practice,
our results shed light on the behavior of different means in high dimensions,
opening the door to future work and a better understanding
other measures of matrix estimation error.  

Our original reason for investigating other matrix means was the problem of
misaligned observations, as happens in brain imaging data.
In such settings, pooling the raw data is not feasible, 
(e.g., the underlying time series of resting state fMRI imaging), since observations cannot be aligned across samples.
We initially believed that $\Har(\cP)$ outperforming $\Ave(\cP)$ empirically was a consequence of this setting, but surprisingly found it to be the case 
even in the classic covariance estimation problem with no misalignment.  
The results presented here are only a partial explanation of this phenomenon,
and a better understanding of the various matrix means in both the aligned and the misaligned settings  in high dimensions warrants further study.

\section*{Acknowledgements}
K. Levin and A. Lodhia were supported by a NSF DMS  Research Training 
Grant 1646108.  E. Levina's research was partially supported by NSF DMS 
grants 1521551 and 1916222.  

\bibliographystyle{plain}
\bibliography{main}

\appendix
\section{Technical Results} \label{sec:multbeta}

\begin{lemma}
\label{lem:integral}
For every $0<a<b$ and $k \geq 1$,
\begin{equation*}
\int_a^b x^{k-1}\sqrt{(b-x)(x-a)}\,\diff x =
\frac{\pi}{2}\bigg(\frac{a+b}{2}\bigg)^{k+1}\sum_{j=0}^{\lfloor\frac{k-1}{2}\rfloor}\frac{1}{j+1}\binom{k-1}{2j} 
\binom{2j}{j}\bigg(\frac{b-a}{b+a}\bigg)^{2j+2}\frac{1}{2^{2j}}.
\end{equation*}
\end{lemma}

\begin{proof}
This identity is a well-known expression for the moments of the
Mar\v{c}enko-Pastur law, aside from changing the support of the distribution.
We include a proof for the sake of completeness.
The change of variables 
\[
u = \frac{2}{b-a}\bigg( x - \frac{a+b}{2}\bigg)
\]
makes the above integral equal to 
\[
\bigg(\frac{b-a}{2}\bigg)^2\int_{-1}^1\bigg\{\bigg(\frac{b-a}{2}\bigg) u + \bigg(\frac{a+b}{2}\bigg)\bigg\}^{k-1}  
\sqrt{1 - u^2}\,\diff u,
\]
and expanding, this is equal to
\[
\sum_{j=0}^{k-1}\binom{k-1}{j} 
\bigg(\frac{b-a}{2}\bigg)^{j+2}\bigg(\frac{a+b}{2}\bigg)^{k-1-j} \int_{-1}^1 u^j \sqrt{1-u^2}\diff u.
\]
Each odd $j$ vanishes by symmetry, yielding
\[
\sum_{j=0}^{\lfloor\frac{k-1}{2}\rfloor}\binom{k-1}{2j}
\bigg(\frac{b-a}{2}\bigg)^{2j+2}\bigg(\frac{a+b}{2}\bigg)^{k-1-2j}
\int_{-1}^1 u^{2j} \sqrt{1-u^2}\diff u.
\]
This integral with respect to $u$
is well-known and can be evaluated either through trigonometric substitution
$u=\sin\theta$ and repeatedly integrating by parts,
or by using evenness, changing variables to $u = \sqrt{v}$
and relating the resulting integral to the Beta function.
We obtain
\[
\int_{-1}^1 u^{2j}\sqrt{1-u^2}\diff u = 
\frac{\pi}{2^{2j+1}}\binom{2j}{j}\frac{1}{j+1}.
\]
Inserting this into the previous expression, the quantity of interest becomes
\[
\frac{\pi}{2}\sum_{j=0}^{\lfloor\frac{k-1}{2}\rfloor}\binom{k-1}{2j}
\bigg(\frac{b-a}{2}\bigg)^{2j+2}\bigg(\frac{a+b}{2}\bigg)^{k-1-2j}\binom{2j}{j}\frac{1}{j+1}\frac{1}{2^{2j}}.
\qedhere
\]
\end{proof}

The following are results from \citep{Kon88}.  Let $\Delta$ be a fixed
$p\times p$ positive definite matrix.

\begin{definition}[Multivariate Beta]
\label{def:multbeta}
A random $p \times p$ positive definite random matrix $L$ is said to
follow a multivariate Beta distribution with parameters $p$, $n_1$, $n_2$ and
$\Delta$ if its density obeys
\begin{equation}
\label{eqn:multbetadens}
\begin{split}
f_L(\ell)&:= K_{n_1,n_2} \frac{\det(\ell)^{\frac{n_1 -p -
      1}{2}}\det(\Delta - \ell)^{\frac{n_2 - p -
      1}{2}}}{\det(\Delta)^{\frac{(N - p - 1)}{2}} }, \quad 0
\preceq \ell \preceq \Delta,\\ K_{n_1,n_2}&:=
\frac{\Gamma_p\big(\frac{N}{2}\big)}{\Gamma_p\big(\frac{n_1}{2}\big)\Gamma_p\big(\frac{n_1}{2}\big)},\\
\Gamma_p(x) &:= \pi^{\frac{p(p-1)}{4}}\prod_{i=1}^p\Gamma\Big(x -
  \frac{i-1}{2}\Big),\\
N &:= n_1 + n_2,
\end{split}
\end{equation}
we denote this distribution as $B(p;n_1,n_2;\Delta)$.
\end{definition}

The following result appears as Corollary 3.3 (ii) in \citep{Kon88}.
We restate it here for ease of reference. 
\begin{theorem}
\label{thm:expofsquare}
Let $T$ be an arbitrary $p\times p$ deterministic matrix.  Suppose the
matrix $L$ is distributed as $B(p;n_1,n_2;\Delta)$ then
\begin{equation}
\BBE[LTL] = \frac{n_1}{N(N - 1)(N+2)} \big[\{ n_1(N+1) -2\} \Delta T
\Delta + n_2\{ (\Delta T \Delta)^\top + \Tr(\Delta T) \Delta \} \big],
\end{equation}
where $N = n_1 + n_2$.
\end{theorem}

\section{Strong Asymptotic Freeness for Real Wishart Matrices}
\label{sec:strongfree}

Let $X_i \in \BBR^{p\times n_i}$ with $1\leq i \leq N$ be a sequence
of random matrices with independent standard Gaussian entries and
define $W_i = X_i X_i^\top/n_i$ as a sequence of i.i.d Wishart
random matrices. Assume that $p/n_i \to \gamma_i > 0$ as
$p\to\infty$. We make use of free probability for what follows, for
an overview see \citep{MiSpe17}. Let $(\cA,\|\cdot\|_\cA,*, \phi)$ be a non-commutative $C^*$-probability space with faithful tracial state $\phi$ such that
there is a sequence of non-commutative freely independent free Poisson
random variables $\{\fPois_i\}_{1\leq i \leq N} \subset \cA$ whose
marginal distributions are
\[
\phi\big(\fPois_j^k\big) = \int_\BBR x^k
\mu_{\mathrm{MP},\gamma_j}(\diff x), \qquad 1 \leq j \leq N, \quad
k\geq 1,
\]
where $\mu_{\mathrm{MP},\gamma_j}$ is the Mar\v{c}enko-Pastur law with
parameter $\gamma_j$.
The space of non-commutative $*$-polynomials of $N$ variables are 
non-commutative polynomials with complex coefficients in variables 
$z_1$, $\ldots\,$, $z_N$ and $z_1^*$, $\ldots\,$, $z_N^*$. We wish to show 
that for any non-commutative $*$-polynomial $Q$ in we have the following 
almost sure convergence
\[
\lim_{n_i\to\infty} \| Q(W_1,\ldots, W_N) \| = \|Q(\fPois_1,\ldots,
\fPois_N)\|_{\cA},
\]
as discussed in the proof of Proposition~\ref{prop:harmlimit}, this is
the only result needed to extend the result of Proposition~\ref{prop:harmlimit} from complex Wishart random matrices to real Wishart random matrices, the
details of which are in \citep[Remark 3]{Lod19}.

This result will the follow from the recent result \citep[Theorem 3.2]{FSW20},
which we restate here. Suppose $\mathfrak{s}_1$, $\ldots\,$,
$\mathfrak{s}_N \in \cA$ are a sequence of freely independent free
semicircular elements, that is,
\[
\phi(\mathfrak{s}_j^k) = \int_\BBR \frac{x^k \sqrt{4-x^2}}{2\pi} \,\diff x,
\]
and $y_1,\ldots, y_q \in \cA$ are fixed with $\{\mathfrak{s}_i\}_{1\leq 
i \leq N}$ free from $\{y_i\}_{1\leq i \leq q}$ and $G_1,\ldots G_N$ 
are an i.i.d sequence of $M\times M$ GOE random matrices (real 
symmetric matrices with centered Gaussian entries with variance $1/M$ off 
the diagonal and $2/M$ on the diagonal) and let $Y_1,\ldots Y_q$ be a 
sequence of $M\times M$ deterministic matrices such that $\sup_{M\geq 
1}\sup_{j\leq q}\|Y_j\| \leq C$ for some constant $C>0$ and for any 
$*$-polynomial $P$ in $q$ non-commutative variables, we have
\begin{align*}
\lim_{M\to\infty}\frac{1}{M}\Tr P(Y_1,\ldots, Y_q) &= \phi\big[
  P(y_1,\ldots, y_q)\big],\\ \lim_{M\to\infty} \|P(Y_1,\ldots, Y_q)\|
&= \|P(y_1,\ldots, y_q)\|_\cA,
\end{align*}
then for any non-commutative polynomial $Q$ in $N+q$ variables, we have 
the almost sure convergence
\begin{align*}
\lim_{M\to\infty}\frac{1}{M}\Tr Q(G_1,\ldots,G_N,Y_1,\ldots Y_q) &=
\phi\big[ Q(\mathfrak{s}_1,\ldots, \mathfrak{s}_N, y_1,\ldots,
  y_q)\big],\\ \lim_{M\to\infty} \|Q(G_1,\ldots, G_N, Y_1,\ldots,
Y_q)\| &= \big\|Q(\mathfrak{s}_1,\ldots, \mathfrak{s}_N, y_1,\ldots,
y_q)\big\|_\cA.
\end{align*}

We now follow the method outlined in \citep[Section 9.2]{Ca12} to 
obtain the result we need from the result of \citep{FSW20} as follows. Let 
$\mathbf{0}_s$ be an $s\times s$ matrix of $0$'s and let 
$\mathbf{0}_{s,t}$ be an $s\times t$ matrix of $0$'s and let the 
parameter $M = p + \sum_{j=1}^N n_j$ and define the sequence of 
$M\times M$ block matrices
\begin{align*}
\BBe_0^{(M)} &= \begin{bmatrix}
\Id_{p} & \mathbf{0}_{p, M-p} \\
\mathbf{0}_{p, M-p} & \mathbf{0}_{M-p}
\end{bmatrix},\\
\BBe_j^{(M)} &= \begin{bmatrix}
\mathbf{0}_{p} &  &  & \\
               &\mathbf{0}_{\sum_{k=1}^{j-1} n_k} &  &\\
               &  & \Id_{n_j} & \\
               &  &  & \mathbf{0}_{\sum_{k=j+1}^N n_k} 
\end{bmatrix} \qquad 1 \leq j \leq N,
\end{align*}
where in the definition of $\BBe_{j}^{(M)}$, $1\leq j \leq N$,
the omitted entries are all 0. Let $G_1,\ldots, G_N$ be a sequence
of i.i.d $M\times M$ GOE matrices. Now consider the matrices
\[
\frac{1}{\sqrt{n_j}}\tilde{X_j} = \sqrt{\frac{M}{n_j}} e_0^{(M)} G_j e_j^{(M)},\qquad 1 \leq j \leq N,
\]
note that the only non-zero entries of the matrix $\tilde{X}_j$ is a 
$p\times n_j$ block matrix in the first $p$ rows and the columns from 
$p + 1 + \sum_{k=1}^{j-1} n_k $ to $p + \sum_{k=1}^j n_k$ and the 
distribution of this block matrix is identical to that of 
$\frac{1}{\sqrt{n_j}} X_j$. We now use the result of the previous 
paragraph to show convergence of $*$-polynomials of 
$\frac{1}{\sqrt{n_j}}\tilde{X_j}$ (with the matrices $\{\BBe_j^{(M)}\}_{1\leq j \leq N}$ playing the role of the deterministic matrices $Y_j$).
The remaining part of the proof of strong freeness of $W_1,\ldots, 
W_N$ follows from  block matrix manipulations, where the polynomials
of $\frac{1}{\sqrt{n_j}}\tilde{X_j}$ are used to produce polynomials in $W_j$. 
These calculations are essentially the same as those described in \citep[Lemmas 9.3--5]{Ca12} which did these manipulations for GUE (complex Hermitian Gaussian matrices) matrices instead of GOE matrices and proved the strong freeness for complex Wisharts as a Corollary to their main result \citep[Theorem 1.6]{Ca12}.
One additional adjustment is needed, since the results of \citep{Ca12} concern Wishart matrices of the form $p = rd$ and $n_j = s_j d$, where $d \to \infty$ and $r$ and $s_j$ are fixed positive integers.
These can be adapted to the present setting in the same way that we adjusted the block matrices $\{\BBe_j^{(M)}\}_{0\leq j \leq N}$ above, and the remaining arguments go through similarly with only cosmetic changes to the original proof.
Details are omitted for the sake of brevity.

\end{document}